\documentclass[11pt]{article}

\usepackage[left=1in,top=1in,right=1in,bottom=1in,nohead]{geometry}

\usepackage{url,subcaption,latexsym,amsmath,amssymb, amsfonts,enumitem, epsfig, graphics,times, amsthm, mathtools, bbm}

\usepackage[normalem]{ulem}

\numberwithin{equation}{section}

\usepackage{comment}

\usepackage{soul}

\usepackage[colorlinks]{hyperref}

\usepackage[usenames,dvipsnames]{xcolor}

\usepackage[english]{babel}

\usepackage{tikz}
\usetikzlibrary{positioning}

\usepackage{graphicx}
\graphicspath{{Plots/}}
\usepackage{caption}
\usepackage{subcaption}
\usepackage{amsmath}
\usepackage{float}

\newtheorem{definition}{Definition}[section]
\newtheorem{theorem}{Theorem}[section]
\newtheorem{corollary}[theorem]{Corollary}
\newtheorem{proposition}[theorem]{Proposition}
\newtheorem{lemma}[theorem]{Lemma}
\newtheorem{remark}[theorem]{Remark}

\def\E{\mathbb{E}}
\def\R{\mathbb{R}}

\def\root{o}
\def\mT{\mathcal T}

\title{Mathematical Analysis for a Class of Stochastic Copolymerization Processes}

\author{
David F. Anderson\thanks{Department of Mathematics, University of
  Wisconsin-Madison, USA.  anderson@math.wisc.edu, grant support from NSF-DMS-2051498.}
\and 
Jingyi Ma\thanks{Department of Mathematics, University of
  Wisconsin-Madison, USA.  jma276@wisc.edu.  Corresponding author.}
\and
Praful Gagrani\thanks{Institute of Industrial Science, The University of Tokyo, Japan. praful@sat.t.u-tokyo.ac.jp }
}
\begin{document}

\maketitle

\begin{abstract} 
     We study a stochastic model of a copolymerization process that has been extensively investigated in the physics literature. The main questions of interest include: (i) what are the criteria for transience, null recurrence,  and positive recurrence in terms of the system parameters; (ii) in the transient regime, what are the limiting fractions of the different monomer types; and (iii) in the transient regime, what is the speed of growth of the polymer? Previous studies in the physics literature have addressed these questions using heuristic methods. Here, we utilize rigorous mathematical arguments to derive the results from the physics literature.  Moreover, the techniques developed  allow us to generalize to the copolymerization process with finitely many monomer types. We expect that the mathematical methods used and developed in this work will also enable the study of even more complex models in the future.
     \\
     
     
\end{abstract}

\section{Introduction}

All known forms of life are composed of cells, which contain long, self-replicating polymers that encode and transmit genetic information. Gaining a comprehensive understanding of the mathematical principles that govern polymer growth involving two or more monomer types (copolymerization) within a well-defined stochastic framework could therefore be essential for understanding the processes underlying the origin of life and the evolution of the genetic code \cite{gagrani2025evolution, koonin2009origin, nowak2008prevolutionary}. Despite considerable progress, a fully developed mathematical formalization of the biologically fundamental copolymer processes—such as DNA replication, wherein a copolymer grows and acquires information guided by another template copolymer—remains an open challenge \cite{gaspard2020template}.

Andrieux and Gaspard \cite{andrieux2008nonequilibrium} were early adopters of a Markovian model of copolymerization, recognizing that the sequence of monomers in the polymer can be described by a continuous-time Markov chain. Thereafter, Esposito, et al.\ \cite{esposito2010extracting} analyzed the thermodynamic efficiency of copolymerization processes using a stochastic kinetic framework, deriving explicit expressions for limiting composition fractions and growth velocity. Their work was grounded in nonequilibrium thermodynamics and relied on entropic arguments, but it did not define the process as a Markov chain nor use formal probabilistic methods in the analysis. 
Similarly, in subsequent work, Gaspard and Andrieux \cite{gaspard2014kinetics}, and later Gaspard alone \cite{gaspard2016kinetics}, developed a framework for these processes and gave explicit expressions for the mean growth velocity and entropy production. While their results were derived analytically, the arguments remained largely heuristic from a mathematical standpoint, relying on thermodynamic consistency and detailed balance identities rather than formal probabilistic arguments.

Building on these developments, in the present work, we revisit this class of models from a mathematical perspective. By recasting the dynamics as a continuous-time Markov process on an infinite tree-like state space, we establish recurrence and transience criteria, and derive almost-sure laws for polymer growth and composition using the theory of Markov chains on trees with finitely many ``cone types'' \cite{woess2009denumerable}. 

In this work, we study a simple copolymerization model in which a set of $d$ monomers, which we will denote throughout via $\mathcal{M} = \{M_1,\dots, M_d\}$,  attach to or detach from the tip of a polymer. This setup reflects a physical constraint: monomers cannot easily insert themselves into the middle of a tightly bound chain. It is also biologically relevant--for example, RNA polymerase extends RNA strands by adding nucleotides ($d=4$) to the 3’ end. 
Despite its simplicity, the model can exhibit quite interesting behavior, especially in the transient regime where the polymer will, with a probability of one, grow without bound.

We also assume that the binding and unbinding rates (affinities) for the different monomers are different, but fixed (i.e., do not depend upon the rest of the polymer chain). This framework can later be extended to address several biologically significant questions. For instance, incorporating sequence-dependent binding affinities allows the model to capture the behavior of template-based polymerization, such as RNA replication \cite{andrieux2008nonequilibrium}. More broadly, a central question in origins-of-life research is whether long polymers can emerge spontaneously or whether ecological interactions are necessary to sustain them \cite{gagrani2025evolution}. Our model provides a principled null model for rigorously exploring such questions.

The organization of the remainder of the paper is as follows.  In Section \ref{sec:math_model}, we introduce the formal mathematical model for the process considered in this paper. Moreover, we more formally state the questions we will address.
In Section \ref{Criterion for Positive Recurrence, Null Recurrence, and Transience}, we establish conditions on the parameters of the model
for when the model is transient, null recurrent, or positive recurrent.  The results of this section are relatively straightforward. 
In Section \ref{Limiting portion of each type of monomer}, we characterize the asymptotic composition of the growing polymer chain in the transient regime. Specifically, for each monomer, $M_i$, we derive the almost sure limiting fraction of that monomer in the growing polymer, as $t \to \infty$.  These fractions will be denoted via $\bar{\sigma}_i$, and are given as functions of the parameter set.
In Section \ref{Asymptotic Growth Rate}, we again consider the transient regime and characterize the rate of growth of the polymer.  Specifically, we establish the existence of a deterministic value $v > 0$, which we derive as a function of the parameter set, such that the polymer length, denoted $|X(t)|$ below, satisfies
   \begin{align*}
       \lim_{t \to \infty} \frac{|X(t)|}{t} = v, \quad \text{almost surely}.
   \end{align*}
   Section \ref{Asymptotic Growth Rate} is the largest part of this paper and contains the bulk of our novel results.
   In Section~\ref{Example: Copolymerization process involving two monomer types}, we restrict to the case of only two monomers (i.e., $d = 2$), which was the setting of our motivating work~\cite{esposito2010extracting}. By restricting our general results to this case, we are able to derive more explicit expressions and provide numerical simulations that help visualize the polymer growth behavior.  This setting not only allows for closed-form analysis, but also serves as a useful comparison for our mathematical treatment with the thermodynamic treatment in~\cite{esposito2010extracting}.

Before proceeding, we explicitly note that throughout this paper we assume a basic knowledge of Markov chains at the level of, for example, the text by Norris \cite{norris1998markov}.

\section{Mathematical model}
\label{sec:math_model}

As mentioned in the introduction, we consider a copolymerization process with finitely many monomer types, $\mathcal{M} = \{M_1,\dots, M_d\}$, with $d \ge 1$. A polymer is then defined as a finite sequence of monomers. Hence, the state space of our model is the set of all finite sequences of monomers, including the polymer consisting of zero monomers, which we denote by $\root$ and refer to as the ``root''.  Thus, if, for example, $d = 3$, the set of polymers includes $\root$, $M_1$, $M_2$, $M_3$, $M_1M_2$, $M_2M_2$, $M_3M_1$, $M_2M_2M_2$, $M_1M_3M_2$, and so forth.  We will denote the state space by $\mT$.  Our resulting continuous-time Markov chain (CTMC) will be denoted by $X$ so that $X(t) \in \mT$ is the state of the process at time $t$.

We turn to the possible transitions of the process. The polymer itself may change in only one of two ways:
\begin{itemize}
    \item[(i)] by having a single monomer of some type, $M_i$, $i \in \{1,\dots,d\}$, attach to the end of the current polymer, or
    \item[(ii)] by having the monomer at the end of the polymer detach.
\end{itemize}
In the first case, if a polymer, denoted $x \in \mT$, has a monomer $M_i$ appended to it, then the new polymer is denoted $x M_i \in \mT$. For example, if $x = M_2M_2M_1M_1M_2$, then $x M_3 = M_2M_2M_1M_1M_2M_3$. Conversely, if the next event is a detachment, then $x M_i$ would transition to $x$.

We now specify the rates of the various transition types. For attachments, as mentioned in the introduction, we assume that the rate at which a monomer $M_i$ is appended to a polymer $x$ depends only on the monomer type, not on the polymer itself. 
We denote these attachment rates by $k_i^+ \in \R_{>0}$, for $i \in \{1, \dots, d\}$. Thus, denoting the transition rates for the process via $q: \mT\times \mT \to \R_{>0}$,  for any $x \in \mT$ and $i \in \{1, \dots, d\}$,
\begin{align*}
    q(x, x M_i) = \lim_{h\to 0^+} \frac{P(X(t+h) = xM_i \ |\ X(t) = x)}{h} = k_i^+, \qquad \text{ for all $t \ge 0$}.
\end{align*}
Similarly, the detachment rates depend only on the identity of the last monomer in the polymer. That is, for appropriate values $k_i^- \in \R_{>0}$, we have
\begin{align*}
    q(x M_i, x) = k_i^-, \quad \text{for } x \in \mT.
\end{align*}
Note that the total rate $q(x)$ out of a state $x$ determines the parameter for the exponential holding time at state $x$. In particular, $q(\root) =\sum_{y\ne \root} q(\root,y) =  \sum_{i = 1}^d k_i^+$ and, for any $x \in \mT$, $q(xM_j)= \sum_{y\ne xM_j} q(xM_j,y) = k_j^- + \sum_{i = 1}^d k_i^+$. Note that these values are uniformly bounded, and hence the process is necessarily non-explosive \cite{norris1998markov}.

We recall that the graph of a Markov chain is defined in the following manner:
\begin{enumerate}
    \item the vertices of the graph are given by the state space;
    \item the directed edges of the graph, denoted as either $(x,y)$ or $x \to y$, for $x,y\in \mathcal{T}$, are determined by the transitions of the chain;
    \item the labels on the edges are determined by the transition rates (in the case of a continuous-time Markov chain) and by the transition probabilities (in the case of a discrete-time Markov chain).
\end{enumerate}

Note that the process we are considering is a continuous-time Markov chain whose graph  is a tree (with root $\root$). (For more on Markov chains on trees, we refer to \cite{woess2009denumerable}.) We will denote the graph of the process by $T$.  For example, in the case of 2 monomers, $M_1$ and $M_2$, the process can be visualized via the graph in Figure \ref{fig:2monomerTree}, with growth progressing downward and detachment corresponding to upward edges. Each vertex represents a polymer (i.e., a finite sequence of monomers), and directed edges correspond to monomer attachment or detachment at the end of the polymer chain. 
Arrows labeled with $k_1^+$ or $k_2^+$ indicate the rate of appending the monomers $M_1$ and $M_2$, respectively, while  arrows labeled with $k_1^-$ or $k_2^-$ represent the rate at which the ending monomer detaches.

\begin{figure}
    \centering
    \begin{tikzpicture}[
  scale=0.55, transform shape,
  node distance=3cm and 3cm,
  every path/.style={->, >=stealth, thick}]

\node (root) at (0,0) {$\root$};
\node (M1) [below left=of root, xshift=+5mm] {$M_1$};
\node (M2) [below right=of root, xshift=-5mm] {$M_2$};
\node (M1M1) [below left=of M1] {$M_1M_1$};
\node (M1M2) [below=of M1] {$M_1M_2$};
\node (M2M1) [below=of M2] {$M_2M_1$};
\node (M2M2) [below right=of M2] {$M_2M_2$};
\node (M1M1M1) [below left=of M1M1, xshift=-10mm] {$M_1M_1M_1$};
\node (M1M1M2) [below=of M1M1, xshift=-12mm] {$M_1M_1M_2$};
\node (M1M2M1) [below=of M1M2, xshift=-20mm] {$M_1M_2M_1$};
\node (M1M2M2) [below right=of M1M2, xshift=-30mm] {$M_1M_2M_2$};
\node (M2M1M1) [below left=of M2M1, xshift=+30mm] {$M_2M_1M_1$};
\node (M2M1M2) [below=of M2M1, xshift=+20mm] {$M_2M_1M_2$};
\node (M2M2M1) [below=of M2M2, xshift=+12mm] {$M_2M_2M_1$};
\node (M2M2M2) [below right=of M2M2, xshift=+10mm] {$M_2M_2M_2$};
\node (dots1) [below=1cm of M1M1M1, rotate=90] {...};
\node (dots2) [below=1cm of M1M1M2, rotate=90] {...};
\node (dots3) [below=1cm of M1M2M1, rotate=90] {...};
\node (dots4) [below=1cm of M1M2M2, rotate=90] {...};
\node (dots5) [below=1cm of M2M1M1, rotate=90] {...};
\node (dots6) [below=1cm of M2M1M2, rotate=90] {...};
\node (dots7) [below=1cm of M2M2M1, rotate=90] {...};
\node (dots8) [below=1cm of M2M2M2, rotate=90] {...};
\path (root) edge node[below right] {$k_1^+$} (M1);
\path (root) edge node[below left]  {$k_2^+$} (M2);
\path (M1) edge node[below right] {$k_1^+$} (M1M1);
\path (M1) edge node[below left]  {$k_2^+$} (M1M2);
\path (M2) edge node[below right] {$k_1^+$} (M2M1);
\path (M2) edge node[below left]  {$k_2^+$} (M2M2);
\path (M1M1) edge node[below right] {$k_1^+$} (M1M1M1);
\path (M1M1) edge node[left]  {$k_2^+$} (M1M1M2);
\path (M1M2) edge node[below right] {$k_1^+$} (M1M2M1);
\path (M1M2) edge node[below left]  {$k_2^+$} (M1M2M2);
\path (M2M1) edge node[below right] {$k_1^+$} (M2M1M1);
\path (M2M1) edge node[below left]  {$k_2^+$} (M2M1M2);
\path (M2M2) edge node[right] {$k_1^+$} (M2M2M1);
\path (M2M2) edge node[below left]  {$k_2^+$} (M2M2M2);
\path (M1) edge[bend left] node[above left] {$k_1^-$} (root);
\path (M2) edge[bend right] node[above right] {$k_2^-$} (root);
\path (M1M1) edge[bend left] node[above left] {$k_1^-$} (M1);
\path (M1M2) edge[bend right] node[right] {$k_2^-$} (M1);
\path (M2M1) edge[bend left] node[left]  {$k_1^-$} (M2);
\path (M2M2) edge[bend right] node[above right] {$k_2^-$} (M2);
\path (M1M1M1) edge[bend left] node[above left] {$k_1^-$} (M1M1);
\path (M1M1M2) edge[bend right] node[right] {$k_2^-$} (M1M1);
\path (M1M2M1) edge[bend left] node[left] {$k_1^-$} (M1M2);
\path (M1M2M2) edge[bend right] node[right] {$k_2^-$} (M1M2);
\path (M2M1M1) edge[bend left] node[left] {$k_1^-$} (M2M1);
\path (M2M1M2) edge[bend right] node[right] {$k_2^-$} (M2M1);
\path (M2M2M1) edge[bend left] node[left] {$k_1^-$} (M2M2);
\path (M2M2M2) edge[bend right] node[above right] {$k_2^-$} (M2M2);

\end{tikzpicture}
    \caption{Reaction graph, $T$, of the copolymerization process involving two monomer types. Each vertex corresponds to a polymer, and edges represent possible transitions due to monomer attachment and detachment. Note the tree-like structure of the graph. }
    \label{fig:2monomerTree}
\end{figure}
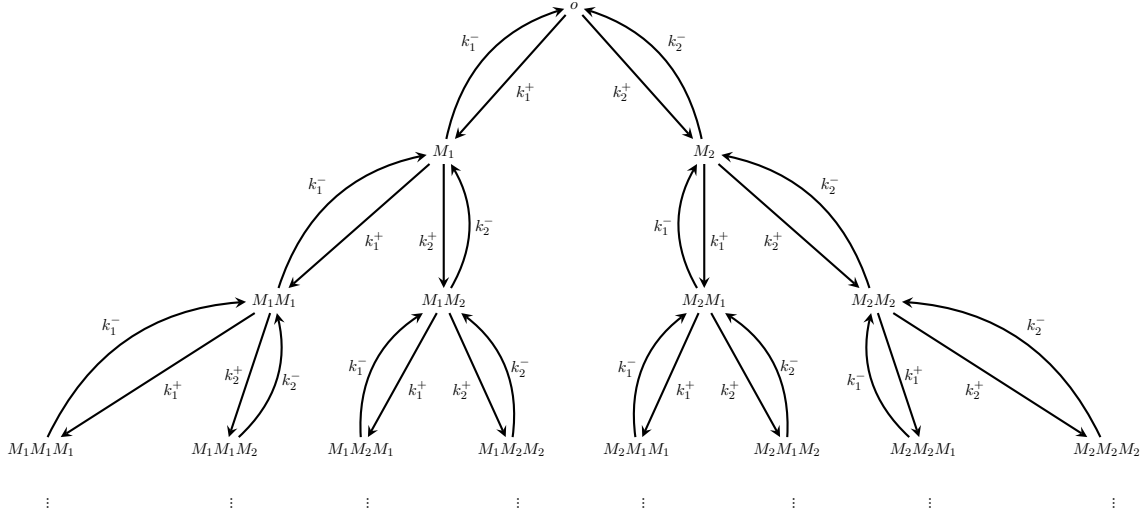

Returning to the general case of $d$ monomers, we write $|x|$ for the length of a polymer $x \in \mT$, i.e., the number of monomers in the polymer. When $|x| \geq 1$, we define the predecessor $x^{-}$ of $x$ to be the unique neighbor of $x$ that is closer to the root, so that $\left|x^{-}\right| = |x| - 1$. For example, if $x = M_1M_2M_3$, then $x^{-} = M_1M_2$.

We denote the embedded discrete-time Markov chain (DTMC) for the process $X$ via $Z$.  Specifically, if we denote $\tau_n$ as the $n$th jump time of the process $X$, with $\tau_0$ taken to be zero, then $Z_n = X(\tau_n)$ \cite{norris1998markov}. In this case, the transition probabilities, $\{p(x,y)\}_{x,y\in \mT}$ of $Z$ satisfy the following:

\begin{itemize}
    \item for $x \in \mT$, $j \in \{1,\dots,d\}$, we have
    \begin{align}
    \label{transition matrix P}
        \begin{split}
        p(xM_j,xM_jM_i) &= \frac{k_i^+}{k_j^-+ \sum_{r=1}^d k_r^+}, \quad \text{for all } i=1,\dots,d, \\
        p(xM_j,x) &= \frac{k_j^-}{k_j^-+\sum_{r=1}^d k_r^+};
        \end{split}
    \end{align}
    \item for the root $\root$, we have
    \begin{align*}
        p(\root, M_i)  = \frac{k_i^+}{\sum_{r=1}^d k_r^+}, \quad \text{for all } i=1,\dots,d.
    \end{align*}
\end{itemize}

After setting up the  model, we can now clearly state the main questions we study in this paper.

\vspace{.1in}

\noindent \textbf{Question 1.} What are the criteria on the parameters $\{k_i^+,k_i^-\}_{i=1}^d$ for when the process $X$ is transient, null recurrent, or positive recurrent? 

\vspace{.1in}

\noindent \textbf{Question 2.}   When the process $X$ is transient, what is the limiting proportion of the $d$ different monomer types, as functions of the parameters $\{k_i^+,k_i^-\}_{i=1}^d$? Specifically, if at time $t$ we denote the length of the polymer by $|X(t)|$, and the number of monomers of type $M_i$ by $N_i^X(t)$, then we want to know if there are values $\bar \sigma_i \in [0,1]$ for which 
    \begin{align*}
        \lim_{t\to \infty} \frac{N_i^X(t)}{|X(t)|} = \bar \sigma_i, \quad \text{for all } i = 1,\dots,d,
    \end{align*}
    almost surely. Moreover, we want to calculate the values $\bar \sigma_i$.

\vspace{.1in}

\noindent \textbf{Question 3.}  When the process $X$ is transient, what is the limiting velocity of the process? Specifically, we would like to know if there is a value $v \in (0,\infty)$ for which 
    \begin{align*}
        \lim_{t\to \infty} \frac{|X(t)|}{t} = v, 
    \end{align*}
    almost surely. Moreover, we want to calculate the value $v$.

Before proceeding to the technical proofs that address Questions 1--3, we briefly comment on the relationship between natural heuristic arguments and the more probabilistic approach developed in this paper. A common first instinct in growth models of this type is to attempt to describe the long-time behavior by writing down formal evolution equations for suitable averaged quantities, with the goal of extracting limiting velocities or asymptotic compositions from such relations. 

However, even at this heuristic level, such arguments inevitably involve the distribution of the terminal (tip) monomer at time $t$, since attachment and detachment events depend only on the current end of the polymer. This tip distribution is not directly determined by the global empirical composition of the polymer, which instead reflects its bulk structure. As a result, these heuristic equations do not form a closed system in terms of the most natural macroscopic observables.

To rigorously relate these two levels of description, we introduce in Section~\ref{Limiting portion of each type of monomer} an auxiliary \emph{boundary process}, which records the terminal monomer types along successive growth events of the polymer. This process allows us to establish almost-sure limits for the empirical composition. In Section~\ref{Asymptotic Growth Rate}, we then use these structural results, together with a careful analysis of the continuous-time dynamics between growth events, to obtain a rigorous characterization of the asymptotic growth velocity. A key technical difficulty is that these growth events do not occur at stopping times for the original process, so their cumulative effect cannot be analyzed by standard renewal or martingale arguments without additional work.

\section{Criterion for positive recurrence, null recurrence, and transience}
\label{Criterion for Positive Recurrence, Null Recurrence, and Transience}

Let $\alpha = \sum_{i = 1}^d \frac{k_i^+}{k_i^-}.$
In this section, we prove that $\alpha$ determines the recurrence properties of the CTMC $X$.  Specifically, we will prove the following theorem.

\begin{theorem}
\label{thm:question1}
 If $\alpha < 1$, then the process $X$ is positive recurrent. If $\alpha = 1$, then $X$ is null recurrent.  If $\alpha > 1$,  then $X$ is transient.
\end{theorem}

To prove the theorem, we first analyze the criteria for recurrence and transience, postponing the distinction between null and positive recurrence until the end. For determining recurrence or transience of the CTMC $\{X(t)\}_{t \geq 0}$, it is sufficient to study the corresponding criteria for the embedded DTMC $\{Z_n\}_{n \in \mathbb{N}}$ \cite[Theorem 3.4.1]{norris1998markov}. This first portion of the proof is essentially an application of the material in \cite[Chapter 9]{woess2009denumerable} (though the second portion, distinguishing between null and positive recurrence, is not).

The plan for the proof of Theorem \ref{thm:question1} is to leverage a particular symmetry of the process.  Specifically, for each monomer, $M_i$, the states of the form $xM_i$, for any $x \in \mT$, are similar in a sense that will be made precise below.  This will allow us to define $d$ different classes, termed ``cone types,'' one for each monomer type $M_i$, $i \in \{1,\cdots,d\}$.  We will then define a matrix $A$ associated with the various cone types, and the spectral radius of $A$ will determine whether the process is recurrent or transient.

We begin with two key definitions:

\begin{definition}
\label{cone}
For $x \in \mT\setminus \{\root\}$, we define 
\begin{align*}
\mT_x &= \{z \in \mT : \text{the first } |x| \text{ monomers of } z \text{ is exactly } x\},
\end{align*}
and $\overline \mT_x = \mT_x \cup\{x^-\}$.
We  define $T_x$ to be the graph with vertices $\overline \mT_x$ and directed edges 
\[
    \{y\to z: y\to z \text{ is an edge of } \ T \text{ and } y \in \mT_x\},
\]
which are precisely the edges with ``starting'' polymer contained within $\mT_x$ (so that $x\to x^-$ is included but $x^-\to x$ is not).  Finally, the labels for the edges of the graph $T_x$ are inherited from $T$.  That is, the label for $y\to z$ in $T_x$ is the same as the label for $y\to z$ in $T$. 
\end{definition}

Note that $T_x$ can be viewed as a subtree rooted at $x$, containing all extensions of $x$, such as $xM_i$, $xM_iM_j$, $xM_iM_jM_k$, and so on.  For technical reasons, it also includes the precursor state $x^-$ and the transition from $x$ to $x^-$.

\begin{definition}
\label{cone type}
The two subtrees  $T_{x}$ and $T_{y}$ are isomorphic if there is a root-preserving bijection between their underlying graphs that preserves edges and labels.  
We will term the isomorphism classes \emph{cone types} and for $x \ne \root$ denote the cone type of $T_x$ as $C(x)$.
\end{definition}

Based on Definitions \ref{cone} and \ref{cone type}, for each monomer type $M_i \in \mathcal{M}$, the associated subtrees  rooted at $xM_i$ share the same cone type, which we denote by $C_i$. Thus, the number of cone types is exactly $d$, one for each monomer type.   Note that $C$ is a function from $\mT$ to $I = \{C_1,\dots,C_d\}$ defined via $C(xM_i) = C_i$ for all $x\in \mT$.  For technical reasons later, we will want $C$ to be defined on the root as well and so we  define $C(\root) = C_0$, but we do not call $C_0$  a cone type. Finally, when referring to the ``cone type of $x$'', we always mean the cone type of the associated subtree  $T_x$.

For a visual example, we again return to the case of two monomers. See Figure~\ref{fig:cone type tree} for a version of Figure~\ref{fig:2monomerTree}, but with two representative subtrees of cone type $C_1$ colored blue and two representative subtrees of cone type $C_2$ colored red. There are, in fact, infinitely many subtrees of each cone type, and each such subtree is infinite (for example, the subtree $T_{M_1}$ is also of cone type $C_1$), but only a small number are shown in the figure for visual clarity.
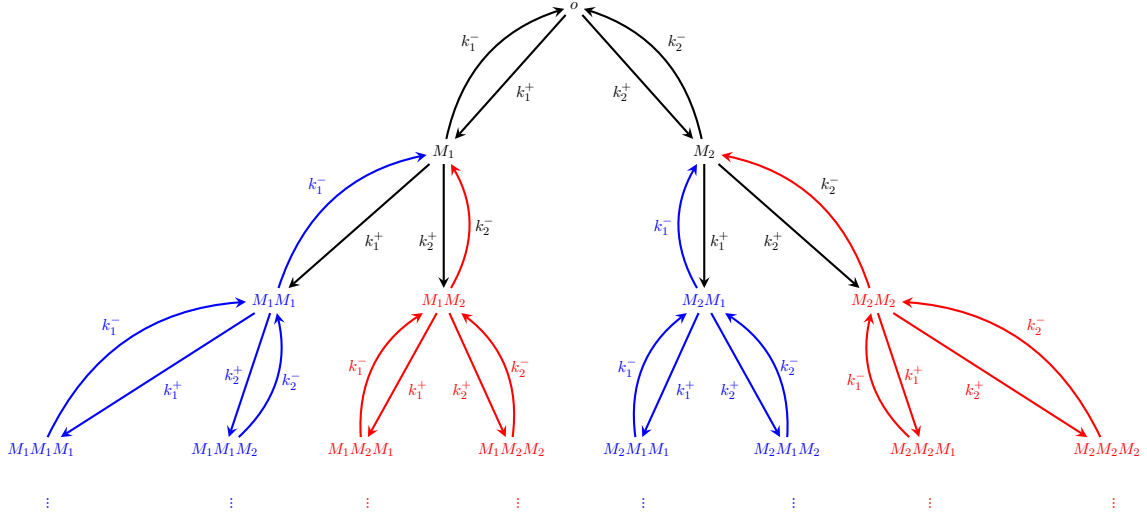
\begin{figure}
    \centering
    \begin{tikzpicture}[
  scale=0.55, transform shape,
  node distance=3cm and 3cm,
  every path/.style={->, >=stealth, thick}]

\node (root) at (0,0) {$\root$};
\node (M1) [below left=of root, xshift=+5mm] {$M_1$};
\node (M2) [below right=of root, xshift=-5mm] {$M_2$};
\node (M1M1) [below left=of M1] {\textcolor{blue}{$M_1M_1$}};
\node (M1M2) [below=of M1] {\textcolor{red}{$M_1M_2$}};
\node (M2M1) [below=of M2] {\textcolor{blue}{$M_2M_1$}};
\node (M2M2) [below right=of M2] {\textcolor{red}{$M_2M_2$}};
\node (M1M1M1) [below left=of M1M1, xshift=-10mm] {\textcolor{blue}{$M_1M_1M_1$}};
\node (M1M1M2) [below=of M1M1, xshift=-12mm] {\textcolor{blue}{$M_1M_1M_2$}};
\node (M1M2M1) [below=of M1M2, xshift=-20mm] {\textcolor{red}{$M_1M_2M_1$}};
\node (M1M2M2) [below right=of M1M2, xshift=-30mm] {\textcolor{red}{$M_1M_2M_2$}};
\node (M2M1M1) [below left=of M2M1, xshift=+30mm] {\textcolor{blue}{$M_2M_1M_1$}};
\node (M2M1M2) [below=of M2M1, xshift=+20mm] {\textcolor{blue}{$M_2M_1M_2$}};
\node (M2M2M1) [below=of M2M2, xshift=+12mm] {\textcolor{red}{$M_2M_2M_1$}};
\node (M2M2M2) [below right=of M2M2, xshift=+10mm] {\textcolor{red}{$M_2M_2M_2$}};
\node (dots1) [below=1cm of M1M1M1, rotate=90, text=blue] {...};
\node (dots2) [below=1cm of M1M1M2, rotate=90, text=blue] {...};
\node (dots3) [below=1cm of M1M2M1, rotate=90, text=red] {...};
\node (dots4) [below=1cm of M1M2M2, rotate=90, text=red] {...};
\node (dots5) [below=1cm of M2M1M1, rotate=90, text=blue] {...};
\node (dots6) [below=1cm of M2M1M2, rotate=90, text=blue] {...};
\node (dots7) [below=1cm of M2M2M1, rotate=90, text=red] {...};
\node (dots8) [below=1cm of M2M2M2, rotate=90, text=red] {...};
\path (root) edge node[below right] {$k_1^+$} (M1);
\path (root) edge node[below left]  {$k_2^+$} (M2);
\path (M1) edge node[below right] {$k_1^+$} (M1M1);
\path (M1) edge node[below left]  {$k_2^+$} (M1M2);
\path (M2) edge node[below right] {$k_1^+$} (M2M1);
\path (M2) edge node[below left]  {$k_2^+$} (M2M2);
\path (M1M1) edge[draw=blue] node[below right, text=blue] {$k_1^+$} (M1M1M1);
\path (M1M1) edge[draw=blue] node[left, text=blue]  {$k_2^+$} (M1M1M2);
\path (M1M2) edge[draw=red] node[below right, text=red] {$k_1^+$} (M1M2M1);
\path (M1M2) edge[draw=red] node[below left, text=red]  {$k_2^+$} (M1M2M2);
\path (M2M1) edge[draw=blue] node[below right, text=blue] {$k_1^+$} (M2M1M1);
\path (M2M1) edge[draw=blue] node[below left, text=blue]  {$k_2^+$} (M2M1M2);
\path (M2M2) edge[draw=red] node[right, text=red] {$k_1^+$} (M2M2M1);
\path (M2M2) edge[draw=red] node[below left, text=red]  {$k_2^+$} (M2M2M2);
\path (M1) edge[bend left] node[above left] {$k_1^-$} (root);
\path (M2) edge[bend right] node[above right] {$k_2^-$} (root);
\path (M1M1) edge[bend left, draw=blue] node[above left, text=blue] {$k_1^-$} (M1);
\path (M1M2) edge[bend right, draw=red] node[right] {$k_2^-$} (M1);
\path (M2M1) edge[bend left, draw=blue] node[left, text=blue]  {$k_1^-$} (M2);
\path (M2M2) edge[bend right, draw=red] node[above right] {$k_2^-$} (M2);
\path (M1M1M1) edge[bend left, draw=blue] node[above left, text=blue] {$k_1^-$} (M1M1);
\path (M1M1M2) edge[bend right, draw=blue] node[right, text=blue] {$k_2^-$} (M1M1);
\path (M1M2M1) edge[bend left, draw=red] node[left, text=red] {$k_1^-$} (M1M2);
\path (M1M2M2) edge[bend right, draw=red] node[right, text=red] {$k_2^-$} (M1M2);
\path (M2M1M1) edge[bend left, draw=blue] node[left, text=blue] {$k_1^-$} (M2M1);
\path (M2M1M2) edge[bend right, draw=blue] node[right, text=blue] {$k_2^-$} (M2M1);
\path (M2M2M1) edge[bend left, draw=red] node[left, text=red] {$k_1^-$} (M2M2);
\path (M2M2M2) edge[bend right, draw=red] node[above right, text=red] {$k_2^-$} (M2M2);

\end{tikzpicture}
    \caption{Reaction graph $T$ of the copolymerization process with two monomer types $M_1$ and $M_2$.  Vertices and edges in blue correspond to the subtrees $T_{M_1M_1}$ and $T_{M_2M_1}$, both having cone type $C_1$,  while those in red correspond to the subtrees $T_{M_1M_2}$ and $T_{M_2M_2}$, both having cone type $C_2$.}
    \label{fig:cone type tree}
\end{figure}

We are in position to prove the main theorem of this section.

\begin{proof}[Proof of Theorem \ref{thm:question1}]
We define  $A$,  a $d\times d$ matrix, whose spectral radius will determine whether the process is transient or recurrent.  
For $i,j \in \{1,\dots, d\}$, we set
\begin{align}
\label{C_i transition}
\begin{split}
    &a(C_i, C_j) = \frac{k_j^+}{k_i^- + \sum_{r=1}^d k_r^+},\quad \text{and} \quad 
    a(C_i^-)=\frac{k_i^-}{k_i^- + \sum_{r=1}^d k_r^+},
\end{split}
\end{align}
where $a(C_i,C_j)$ is the transition probability from a state $xM_i$ to $xM_iM_j$, and $a(C_i^-)$ is the probability of moving from $xM_i$ to $x$ (compare with \eqref{transition matrix P}).  Then, for $i, j \in \{1,\dots, d\}$, we define  (see formula (9.77) in \cite{woess2009denumerable})
\begin{align}
\label{matrix A def}
    A_{ij} = \frac{a(C_i, C_j)}{a(C_i^{-})} = \frac{k_j^+}{k_i^-}.
\end{align}

Therefore, the matrix $A$ takes the form:
\begin{align*}
    A =
\left( \frac{k_j^+}{k_i^-} \right)_{1 \leq i,j \leq d} =
\left(\begin{array}{cccc}
    \frac{k_1^+}{k_1^-} & \frac{k_2^+}{k_1^-} & \cdots & \frac{k_d^+}{k_1^-} \\[1ex]
    \frac{k_1^+}{k_2^-} & \frac{k_2^+}{k_2^-} & \cdots & \frac{k_d^+}{k_2^-} \\[1ex]
    \vdots & \vdots & \ddots & \vdots \\[1ex]
    \frac{k_1^+}{k_d^-} & \frac{k_2^+}{k_d^-} & \cdots & \frac{k_d^+}{k_d^-}
\end{array}\right).
\end{align*}
Observe that $A$ is a rank-one matrix of the form $A = \mathbf{u} \mathbf{v}^T$, where
\begin{align*}
    \mathbf{u} = \left( \frac{1}{k_1^-}, \frac{1}{k_2^-}, \dots, \frac{1}{k_d^-} \right)^T, \quad
    \mathbf{v} = \left( k_1^+, k_2^+, \dots, k_d^+ \right)^T.
\end{align*}
Such a matrix has one nonzero eigenvalue equal to the inner product $\mathbf{v}^T \mathbf{u} = \sum_{i=1}^d \frac{k_i^+}{k_i^-} > 0$, and the remaining $d-1$ eigenvalues are all zero. Hence, the spectral radius is
\begin{align*}
    \alpha = \sum_{i=1}^d \frac{k_i^{+}}{k_i^{-}}.
\end{align*}
Therefore, according to Theorem 9.78 in \cite{woess2009denumerable} and Theorem 3.4.1 in \cite{norris1998markov}, we may conclude the following.
\begin{itemize}
    \item If $\sum_{i=1}^d \frac{k_i^{+}}{k_i^{-}} \leq 1$, then the DTMC $Z$, and hence the CTMC $X$, is recurrent;
    \item If $\sum_{i=1}^d \frac{k_i^{+}}{k_i^{-}} > 1$, then the DTMC $Z$, and hence the CTMC $X$, is transient.
\end{itemize}

We now distinguish between positive and null recurrence for $X$. The process
$X$ is positive recurrent if and only if there exists a unique probability distribution $\{\mu(x) : x \in \mT\}$ satisfying the global balance equations
\begin{align}
\label{eq:balance}
    q(x)\mu(x) = \sum_{y \ne x} \mu(y) q(y, x),
\end{align}
where $q(x) = \sum_{y \ne x} q(x, y)$ is the total exit rate from state $x$ (see, for example, \cite[Theorem 3.5.3]{norris1998markov}). This characterization assumes that the process is non-explosive, which holds in our setting because the total jump rate from any state, namely $q(x)$, is uniformly bounded above (see \cite{norris1998markov}). 

We now define a measure $\mu : \mT \to \mathbb{R}_{\ge 0}$ by
\begin{align}
\label{balance equations}
    \mu(x) := \mu(\root) \prod_{i=1}^{d} \left( \frac{k_i^{+}}{k_i^{-}} \right)^{\beta_i(x)},
\end{align}
where $\beta_i(x)$ denotes the number of monomers of type $M_i$ in polymer $x$, and $\mu(\root)$ is a normalizing constant.

\begin{proposition}
The measure defined in \eqref{balance equations} satisfies the balance equations~\eqref{eq:balance}.
\end{proposition}

\begin{proof}
We simply check that the balance equation \eqref{eq:balance} holds for each state $x$. We make use of the fact that for any $x \in \mT$ and $M_i \in \mathcal{M}$, 
\begin{align}
\label{balance equation relationship}
    \mu(x M_i) = \mu(x) \left(\frac{k_i^{+}}{k_i^{-}}\right).
\end{align}

\begin{itemize}
    \item We begin by  verifying \eqref{eq:balance} for the root, $x = \root$.  Since $q(\root) =  \sum_{i=1}^{d} k_i^{+}$, we have $q(\root)\mu(\root) = \left( \sum_{i=1}^{d} k_i^{+} \right) \mu(\root).$
    Moreover, 
    \begin{align*}
        \sum_{y \neq \root} \mu(y) q(y, \root) &= \sum_{i=1}^{d} \mu(M_i) q(M_i, \root)   \\
        &= \sum_{i=1}^{d} \mu(\root) \left(\frac{k_i^{+}}{k_i^{-}}\right) \cdot k_i^{-} \tag{using \eqref{balance equation relationship} and that $q(M_i,0) = k_i^-$}\\
        &= \left( \sum_{i=1}^{d} k_i^{+} \right) \mu(\root). 
    \end{align*}
    Hence, $q(\root)\mu(\root) = \sum_{y \neq \root} \mu(y) q(y, \root)$.

    \item  We now consider states of the form $xM_j$, with $x \in \mT$ and $M_j \in \mathcal{M}$.  We  first note that \eqref{balance equation relationship} yields
    \begin{align*}
        q(x M_j)\mu(x M_j) = \left( k_j^{-}+\sum_{i=1}^{d} k_i^{+} \right) \mu(x) \left(\frac{k_j^{+}}{k_j^{-}}\right).
    \end{align*}
    Next, we have
    \begin{align*}
        \sum_{y \neq x M_j} \mu(y) q(y, x M_j)   &= \mu(x) q(x, x M_j) + \sum_{i=1}^{d} \mu(x M_j M_i) q(x M_j M_i, x M_j) \\
        &= \mu(x) k_j^{+} + \sum_{i=1}^{d} \mu(x) \left(\frac{k_j^{+}}{k_j^{-}}\right) \left(\frac{k_i^{+}}{k_i^{-}}\right) \cdot k_i^{-} \tag{using \eqref{balance equation relationship} twice}\\
        &= \mu(x) k_j^{+} + \mu(x) \left(\frac{k_j^{+}}{k_j^{-}}\right) \sum_{i=1}^{d} k_i^{+} \\
        &= \left( k_j^{-}+\sum_{i=1}^{d} k_i^{+} \right) \mu(x) \left(\frac{k_j^{+}}{k_j^{-}}\right).
    \end{align*}
    Therefore, $q(x M_j)\mu(x M_j) = \sum_{y \neq x M_j} \mu(y) q(y, x M_j)$. Since both sides match, the proposition has been proved.\qedhere
\end{itemize}
\end{proof}

Now we need to give the condition under which $\left\{\mu_x: x \in \mT \right\}$ forms a probability distribution. This requires the measure sums to 1. 
Note that the number of polymers of length $\ell$ that consist of $\beta_i$ monomers of type $M_i$ (so that $\beta_1 + \dots + \beta_d = \ell$) is precisely the multinomial coefficient $\binom{\ell}{\beta_1, \dots, \beta_d}$. This accounts for the number of distinct sequences (i.e., orderings) of monomers with those multiplicities. Therefore,
\begin{align*}
   \sum_{x \in \mT} \mu(x) &= \mu(\root) + \sum_{\ell=1}^{\infty} \sum_{\substack{\beta_1 + \cdots + \beta_d = \ell \\ \beta_i \geq 0}} \mu(\root) \binom{\ell}{\beta_1, \dots, \beta_d} \prod_{i=1}^d \left( \frac{k_i^{+}}{k_i^{-}} \right)^{\beta_i} \\
   &= \mu(\root) \left(1 + \sum_{\ell=1}^{\infty} \left( \sum_{i=1}^d \frac{k_i^{+}}{k_i^{-}} \right)^l \right),
\end{align*}
where the final equality follows from the multinomial theorem.

Hence, $\mu(0)$ can be chosen for $\sum_{x \in \mT} \mu(x)$ to equal one if and only if $\alpha = \sum_{i=1}^d \frac{k_i^+}{k_i^-} < 1$. 

Hence, if $\alpha = \sum_{i=1}^d \frac{k_i^{+}}{k_i^{-}} < 1$, then the process $X$ is positive recurrent.

Now consider the case where $\alpha = \sum_{i=1}^d \frac{k_i^{+}}{k_i^{-}} = 1$. The system still admits a unique stationary measure $\{\mu(x): x \in \mT\}$ given by~\eqref{balance equations} because stationary measures for irreducible recurrent continuous-time Markov chains are unique up to scalar multiples (see \cite[Theorem 3.5.2]{norris1998markov}). However, in the case $\alpha = 1$  this measure cannot be normalized to a probability distribution.  Hence, in this case,  the process $X$ cannot be positive recurrent, and so must be null recurrent, concluding the proof.
\end{proof}

\section{Limiting proportion of each monomer type}
\label{Limiting portion of each type of monomer}

We now turn to our second question. Throughout this section we assume that $\alpha = \sum_{i=1}^d \frac{k_i^{+}}{k_i^{-}} > 1$,
so that the process $X$ is transient and   $\lim_{t \to \infty} \left|X_t\right| = \infty$, almost surely \cite[Theorem 9.18]{woess2009denumerable}. For each $i \in \{1,\dots,d\}$ and any $t \ge 0$, we denote the number of occurrences of the monomer $M_i$ in the polymer $X(t)$ by $N_i^X(t)$.  Using the notation of the last section (in \eqref{balance equations}), we note $N_i^X(t) = \beta_i(X(t))$. 
The proportion of monomer $M_i$ at time $t$ is then
\begin{align}
\label{eq:98769877689}
    \sigma_i(t) :=\frac{N_i^X(t)}{\sum_{j=1}^d N_j^X(t)}= \frac{N_i^X(t)}{|X(t)|},
\end{align}
with each $\sigma_i(t)$ taken to be zero when $X(t) = \root$.

In this section, we prove the following. 

\begin{theorem}
    \label{thm:proportion}
    In the transient regime, i.e, when $\alpha = \sum_{i=1}^d \frac{k_i^{+}}{k_i^{-}} > 1$, for each $i\in \{1,\dots,d\}$ we have $\lim_{t\to \infty} \sigma_i(t) = \bar \sigma_i$, almost surely, where 
    \begin{align*}
        \bar \sigma_i = \frac{k_i^+}{m+k_i^-},
    \end{align*}
    with $m$ being the unique value satisfying $\sum_{r=1}^d \frac{k_r^+}{m+k_r^-} = 1$.
\end{theorem}

\begin{remark}
Although we do not use this interpretation in the analysis, the constant $m$ admits a natural probabilistic meaning as an effective \emph{escape rate} from the root. In particular, one can show that
\[
m = \left(\sum_{r=1}^d k_r^+\right)
P_{\root}\left(\text{$X$ never returns to $\root$}\right).
\]
Thus, $m$ may be viewed as the total rate at which growth attempts are initiated at the root, multiplied by the probability that such an attempt leads to a trajectory that never returns.  Equivalently, $m$ may be interpreted as the effective drift of the process away from the root in the transient regime.
\end{remark}

To prove the theorem, it suffices to study the embedded discrete-time Markov chain $\{Z_n\}_{n \ge 0}$. Moreover, and without loss of generality, we will assume throughout this section that the process has an initial state given by the root; that is, $Z_0 = \root$ with probability one.

We begin by defining some of the key objects for the next two sections.  First, 
we define the $k$th \textit{level} of the graph $T$ to be the subset of the state space $\mT$ consisting of polymers with length $k$. For example, when $d = 2$, the second level is the set $\{M_1M_1,M_1M_2,M_2M_1,M_2M_2\}$. Next, the  random times $\{e_k\}_{k \ge 0}$ are defined to be the \textit{last} time the process $Z$ visits level $k$.  That is,
\begin{align*}
    e_k = \sup \left\{ n \geq 0 : \left|Z_n\right| = k \right\}.
\end{align*}
Note that because the process is assumed to be transient, we have $e_k < \infty$ for each $k$, with probability one, and that  for any $k \geq 0$, we have $|Z_{e_k}| = k$ and $|Z_n| > k$ for all $n > e_k$. Note also that the $e_k$ are \textit{not} stopping times.

We now construct the process $\{W_k\}_{k \ge 0}$, sometimes referred to as the \textit{boundary process} \cite{woess2009denumerable}. For each $k \ge 0$, we set
\begin{align}
\label{bbboundary Process}
    W_k = Z_{e_k},
\end{align}
which records the state visited at the last time the process $Z$ is at level $k$. 
 
It follows that $W_k$ is the polymer of length $k$ that forms the first $k$ monomers of the limiting infinite polymer.  In particular, note that $(W_0,W_1,W_2,\dots, W_k)$ converges, as $k \to \infty$, to an infinite length polymer, and that the fractional representation of each monomer in the $W$ process is the object of our interest in this section. For readers who wish to see a concrete example immediately, a visualization of the boundary process in the two-monomer case appears in Section \ref{Example: Copolymerization process involving two monomer types}.

The plan is the following.  According to Theorem \ref{thm:WisMC} below, the process $\{W_k\}_{k\ge 0}$ is itself a Markov chain.  Define the associated cone type of $W_k$ to be $U_k$.  That is,
\begin{align}
\label{conecone boundary}
    U_k = C(W_k).
\end{align}
The process $U_k$ is then a Markov chain on the finite state space $\{C_1,\dots, C_d\}$.  We will prove below that $U_k$ is irreducible.  Hence, it has a unique limiting stationary distribution.  Moreover, this  distribution  yields the desired limiting proportion of each monomer type.   Thus, our remaining goal is to characterize the limiting (stationary) distribution of the  process $U_k$.

Our first order of business is to characterize the transition probabilities for  $\{W_k\}_{k\ge 0}$. To that end, for any $x, y \in \mT$, define $ f^{(n)}(x,y) = 0$ and for $n \ge 1$,
\begin{align*}
    f^{(n)}(x,y) = P_x \left( Z_n = y, \ Z_m \neq y \text{ for } 0 < m < n \right),
\end{align*}
 which is the probability that the first time the process $Z$ enters state $y$ is at time $n$ (after $n$ jumps), given that the process $Z$ starts at state $x$.  We then define
\begin{align}
\begin{split}
    F(x,y)  &:= \sum_{n=0}^\infty f^{(n)}(x,y)= P_x(\text{the process enters state } y \text{ in finite time}).
\end{split}
\label{Sec4:F}
\end{align}
It is intuitively clear that for any monomer type $M_i \in \mathcal{M}$ and any $x \in \mT$, the value $F(x M_i, x)$
only depends on the cone type $C(x M_i)= C_i$.  
(For a reference to this fact, see Chapter 9, page 276 in \cite{woess2009denumerable}.) 
Hence, for each $M_i \in \mathcal{M}$ and any $x \in \mT$, we denote
\begin{align}
\label{eq:F_i_sec4}
    F_i := F(x M_i, x).
\end{align}
 Note  that each $F_i$ is strictly greater than zero (and, in fact, lower bounded by $\frac{k_i^-}{k_i^- + \sum_{r=1}^d k_r^+}$) and is also strictly less than one \cite[Lemma 9.98]{woess2009denumerable}.

We now state the following key result from \cite{woess2009denumerable}.

\begin{theorem}
\label{thm:WisMC}
  In the transient case, i.e., $\alpha > 1$,  the  process $\left\{W_k\right\}_{k \geq 1}$ is  a Markov chain. For $x \in \mT$ with $|x| = k$ and any $y \in \mT$ for which $x=y^-$, we have the following transition probabilities
  \begin{align*}
      P\left(W_k = y \mid W_{k-1} = x\right) = F(x, x^{-}) \cdot \frac{1 - F(y, x)}{1 - F(x, x^{-})} \cdot \frac{p(x, y)}{p(x, x^{-})},
  \end{align*}
  where the $p$ are given in and around \eqref{transition matrix P}.
\end{theorem}

From this, we can immediately calculate the transition probabilities for $W_k$ in terms of the $F_i$ in \eqref{eq:F_i_sec4}. In particular, for the polymers $xM_i$ and $xM_iM_j$, with $x \in \mT$ and $|x|=k-1$,
\begin{align}
P\left(W_{k+1} = x M_i M_j \mid W_k = x M_i\right)
&= F_i \cdot \frac{1 - F_j}{1 - F_i} \cdot \frac{p(x M_i, x M_i M_j)}{p(x M_i, x)} \notag\\
&= F_i \cdot \frac{1 - F_j}{1 - F_i} \cdot \frac{\frac{k_j^{+}}{k_i^{-} +\sum_{r=1}^d k_r^{+} }}{\frac{k_i^{-}}{ k_i^{-} + \sum_{r=1}^d k_r^{+} }} \notag\\
&= F_i \cdot \frac{1 - F_j}{1 - F_i} \cdot \frac{k_j^{+}}{k_i^{-}}.\label{eq:7896675}
\end{align}
Note that each term is well defined because $F_i <1$  and that each term is also strictly positive.  
We then immediately conclude that the process $\{U_k\}_{k \ge 1}$ is irreducible and has the following transition probabilities
\begin{align*}
    P(U_{k+1} = C_j \mid U_k = C_i) = F_i \cdot \frac{1 - F_j}{1 - F_i} \cdot \frac{k_j^{+}}{k_i^{-}}.
\end{align*}

We can now give a $d \times d$ transition matrix $V = (V_{ij})_{1 \le i,j \le d}$ for the Markov chain $\{U_k\}_{k \ge 1}$: 
\begin{align}
\label{general_M_matrix}
\begin{split}
V_{ij} :=  P(U_{k+1} = C_j \mid U_k = C_i) 
= F_i \cdot \frac{1 - F_j}{1 - F_i} \cdot \frac{k_j^{+}}{k_i^{-}}>0.
\end{split}
\end{align}
This matrix $V$ indicates $\{U_k\}_{k \ge 1}$ is irreducible and positive recurrent (under our transient assumption). Let $\bar{\sigma} = (\bar{\sigma}_1, \dots, \bar{\sigma}_d)$ denote the stationary distribution of this Markov chain  $\{U_k\}_{k \ge 1}$. Then $\bar{\sigma}$ satisfies:
\begin{align}
    \bar{\sigma} V = \bar{\sigma}, \quad \sum_{i=1}^d \bar{\sigma}_i = 1.
    \label{stationary_general}
\end{align}

Thus, all that remains is to calculate the $F_i$ of \eqref{eq:F_i_sec4} and derive the stationary distribution for the process.

Before that, we give the following propositions for preparation.

\begin{proposition}
    The $F_i$ of \eqref{eq:F_i_sec4} satisfy the following system of $d$ equations,
    \begin{align}
    \label{Properties of F d-types}
        \frac{k_i^{-}}{\left( k_i^{-} + \sum_{r=1}^{d} k_r^{+}\right) - \sum_{r=1}^{d} k_r^{+} F_r} = F_i, \qquad \text{for each } i \in \{1, \dots, d\}.
    \end{align}
\end{proposition}

The proof of the above proposition can be found in and around \cite[Equation 9.76]{woess2009denumerable}.

\begin{proposition}
    The solution for \eqref{Properties of F d-types} exists and is equal to
    \begin{align*}
        F_i = \frac{k_i^{-}}{m+k_i^{-}}, \quad i \in \{1,\cdots,d\}
    \end{align*}
    where $m$ is the unique solution satisfying
    \begin{align*}
        \sum_{r=1}^{d} \frac{k_r^+}{m+k_r^{-}} =1.
    \end{align*}
\label{F_i solutions}
\end{proposition}

\begin{remark}
    When the number of monomer types satisfies $d<5$, this equation can be solved analytically by reducing it to a polynomial of degree at most four. However, for $d \geq 5$, the equation is not in general solvable in radicals due to the Abel-Ruffini theorem.  However, in that case the value of $m$ can be computed numerically.
\end{remark}

\begin{proof}[Proof of Proposition \ref{F_i solutions}]
    Manipulating \eqref{Properties of F d-types} shows that for any $i\in \{1,\dots,d\}$,
    \begin{align}
    \label{SABABABAAB}
        &\frac{k_i^{-}}{F_i}-k_i^{-}= \sum_{r=1}^{d} k_r^{+}  - \sum_{r=1}^{d} k_r^{+} F_r.
    \end{align}
    Denote
    \begin{align}
    \label{m def}
        m=\sum_{r=1}^{d} k_r^{+}  - \sum_{r=1}^{d} k_r^{+} F_r = \sum_{r = 1}^d k_r^+(1-F_r),
    \end{align}
    and note that $m>0$ because each $F_r\in (0,1)$.
    Combining \eqref{SABABABAAB} and \eqref{m def} yields
    \begin{align}
    \label{F_i and m}
        &\frac{k_i^{-}}{F_i}-k_i^{-}= m \quad \implies \quad F_i = \frac{k_i^{-}}{m+k_i^{-}}.
    \end{align}
    Plugging this back into \eqref{m def} yields
    \begin{align*}
        &m=\sum_{r=1}^{d} k_r^{+}  - \sum_{r=1}^{d} k_r^{+} F_r = \sum_{r=1}^{d} k_r^{+}  - \sum_{r=1}^{d} k_r^{+} \left(\frac{k_r^{-}}{m+k_r^{-}}\right) = \sum_{r=1}^{d} k_r^{+} \frac{m}{m+k_r^{-}}= m \sum_{r=1}^{d} k_r^{+} \frac{1}{m+k_r^{-}}.
    \end{align*}
    Since $m>0$, we may divide by $m$ and conclude
    \begin{align*}
        \sum_{r=1}^{d} \frac{k_r^+}{m+k_r^{-}} =1.
    \end{align*}
    Note that the value of $m$ satisfying the above is unique because the function $g(m) = \sum_{r=1}^{d} \frac{k_r^+}{m+k_r^{-}}$ is monotonically decreasing in $m$, $g(0)=\sum_{r=1}^{d} \frac{k_r^+}{k_r^{-}}>1$, and approaches zero as $m \to \infty$.
\end{proof}

We can now give the stationary distribution for the process $U$.

\begin{proposition}
The stationary distribution for the process $\{U_k\}_{k \ge 1}$ is 
    \begin{align}
    \label{sigma and F relationship}
    \begin{split}
        \bar{\sigma}&= (\bar \sigma_1, \dots, \bar \sigma_d)=
        \left(\frac{k_1^{+}}{m+k_1^{-}},\frac{k_2^{+}}{m+k_2^{-}},\cdots,\frac{k_d^{+}}{m+k_d^{-}}\right),
    \end{split}
    \end{align}
    where $m$ is the unique solution to
    \begin{align*}
        \sum_{r=1}^{d} \frac{k_r^+}{m+k_r^{-}} =1.
    \end{align*}
\end{proposition}

\begin{proof}
    From \eqref{SABABABAAB},
    \begin{align*}
        &\frac{k_i^-(1-F_i)}{F_i} =
        \sum_{r=1}^{d} k_r^{+}  - \sum_{r=1}^{d} k_r^{+} F_r.
    \end{align*}
    Note that the right-hand side of the above does not depend upon $i$.  Thus, for any $i,j\in\{1,\dots,d\}$, we have $\frac{k_i^-(1-F_i)}{F_i} = \frac{k_j^-(1-F_j)}{F_j}$. Hence,
    \begin{align*}
       \frac{1-F_j}{1-F_i}=\frac{k_i^- F_j}{k_j^- F_i}. 
    \end{align*}
    Plugging this into \eqref{general_M_matrix} yields
    \begin{align*}
        V_{ij} := F_i \cdot \frac{1 - F_j}{1 - F_i} \cdot \frac{k_j^{+}}{k_i^{-}}= F_i \cdot \frac{k_i^- F_j}{k_j^- F_i} \cdot \frac{k_j^{+}}{k_i^{-}}= F_j \frac{k_j^+}{k_j^-}.
    \end{align*}
    Since $V$ is a transition matrix,
    \begin{align}
    \label{vector 1}
        \sum_{j=1}^{d} F_j \frac{k_j^+}{k_j^-}=1,
    \end{align}
    and so for $i \in \{1,2,\cdots,d\}$, 
    \begin{align}
    \label{vector 2}
        F_i \frac{k_i^+}{k_i^-}= F_i \frac{k_i^+}{k_i^-} \left(\sum_{j=1}^{d} F_j \frac{k_j^+}{k_j^-}\right) = \sum_{j=1}^{d} F_i \frac{k_i^+}{k_i^-} F_j \frac{k_j^+}{k_j^-} = \sum_{j=1}^{d} V_{ji}  \left(F_j \frac{k_j^+}{k_j^-}\right).
     \end{align}
    Hence, according to \eqref{vector 1} and \eqref{vector 2} and Proposition \ref{F_i solutions}, the stationary distribution is
    \begin{align}
    \label{stationary dist}
        \bar{\sigma}_i=F_i \frac{k_i^{+}}{k_i^{-}}=\frac{k_i^{+}}{m+k_i^{-}}, \quad  i \in \{1,2,\cdots,d\},
    \end{align}
    where $m$ is the unique solution given by $\sum_{r=1}^{d} \frac{k_r^+}{m+k_r^{-}} =1.$
\end{proof}

At this point, we have determined the stationary distribution of the chain $\{U_k\}_{k\ge 1}$, which describes the limiting frequency of cone types along the boundary process $W_k$. Intuitively, this already suggests that the limiting proportion of each monomer type in the polymer should be given by \eqref{sigma and F relationship}. However, the connection is not yet completely rigorous: the limiting frequencies of cone types in the boundary process must be related back to the original proportion $\sigma_i(t)$ of \eqref{eq:98769877689}  for the process $X$. 
Specifically, if we denote the number of occurrences of the monomer $M_i$ in the polymer $W_k$ by $N_i^W(k)$, we now know that from ergodic theorem, almost surely,
\begin{align*}
    \lim_{k\to\infty} \frac{N_i^W(k)}{|W_k|}=\lim _{n \rightarrow \infty} \frac{1}{n} \sum_{k=1}^n \mathbf{1}_{C_i}\left(C\left(W_k\right)\right)=\bar{\sigma}_i,
\end{align*}
all that remains is to show 
\begin{align*}
    \lim_{t\to\infty} \sigma_i(t)
= \lim_{k\to\infty} \frac{N_i^W(k)}{|W_k|},
\end{align*}
almost surely.  This requires carefully embedding the continuous-time process into the discrete $W$ process. The remainder of the proof is devoted to establishing this connection.

\begin{proof}[Proof of Theorem \ref{thm:proportion}]
 Define
\begin{align*}
    J_t = \max\{n \in \mathbb{N} \mid \tau_n \le t\},
\end{align*}
to be the number of jumps of $X$ up to time $t$. Since $X_t = Z_{J_t}$,
\begin{align*}
     \lim_{t \to \infty} \sigma_i(t)
    =\lim_{t \to \infty} \frac{N_i^X(t)}{|X(t)|}
    = \lim_{t \to \infty} \frac{N_i^Z(J_t)}{|Z_{J_t}|},
\end{align*}
where the number of occurrences of the monomer $M_i$ in the polymer $Z_n$ by $N_i^Z(n)$.

Because $J_t \to \infty$ a.s. \cite{norris1998markov}, it follows that 
\begin{align}
    \lim_{t \to \infty} \sigma_i(t)
    = \lim_{n \to \infty} \frac{N_i^Z(n)}{|Z_n|}, \quad \text{with probability 1}.
    \label{eq:step1}
\end{align}
We need to embed this limit onto $W$. 
As above, let $e_k$ be the last time $Z$ visits level $k$ and define
\begin{align*}
    \boldsymbol{\hat k}(n) = \max\{k : e_k \le n\},
\end{align*}
giving the length of the boundary process at time $n$.
Then $\boldsymbol{\hat k}(n) \to \infty$ as $n\to \infty$ a.s. \cite[Page 295]{woess2009denumerable}, and so
\begin{align}
    \lim_{n \to \infty} \frac{N_i^W(\boldsymbol{\hat k}(n))}{|W_{\boldsymbol{\hat k}(n)}|}
    = \lim_{k \to \infty} \frac{N_i^W(k)}{|W_k|}, \quad \text{ with probability 1}.
    \label{hddjchdjcdf}
\end{align}
Combining \eqref{eq:step1} and \eqref{hddjchdjcdf}, we now simply need to show the following holds almost surely,
\[
\lim_{n \to \infty} \frac{N_i^W(\boldsymbol{\hat k}(n))}{|W_{\boldsymbol{\hat k}(n)}|} = \lim_{n \to \infty} \frac{N_i^Z(n)}{|Z_n|}.
\]
To that end, we decompose the right hand side (where we assume $n$ is large enough so that none of the denominators are zero),
\begin{align}
    \frac{N_i^Z(n)}{|Z_n|}
    = \frac{N_i^Z(n)}{N_i^W(\boldsymbol{\hat k}(n))}
       \cdot \frac{N_i^W(\boldsymbol{\hat k}(n))}{|W_{\boldsymbol{\hat k}(n)}|}
       \cdot \frac{|W_{\boldsymbol{\hat k}(n)}|}{|Z_n|}.
       \label{fhfjcdc}
\end{align}
We will show that the first and third ratios limit to 1, almost surely, in which case we are done.

We tackle the first ratio.  First note that $Z_n \in \mT_{W_{\boldsymbol{\hat k}(n)}}$ (i.e., the first $\hat{k}(n)$ monomers of $Z_n$ coincide with $W_{\hat{k}(n)}$).  Hence, for each choice of $i$ and $n$,
\[
1 \le \frac{N_i^Z(n)}{N_i^W(\boldsymbol{\hat k}(n))}.
\]
We also have an upper bound,
\begin{align*}
\frac{N_i^Z(n)}{N_i^W(\boldsymbol{\hat k}(n))}
    &= \frac{N_i^Z(n) - N_i^W(\boldsymbol{\hat k}(n)) + N_i^W(\boldsymbol{\hat k}(n))}{N_i^W(\boldsymbol{\hat k}(n))}\\
    &\le \frac{|Z(n)| - |W(\boldsymbol{\hat k}(n))| + N_i^W(\boldsymbol{\hat k}(n))}{N_i^W(\boldsymbol{\hat k}(n))} \tag{since $Z_n \in \mT_{W_{\boldsymbol{\hat k}(n)}}$}\\
    &\le \frac{ e_{\boldsymbol{\hat k}(n)+1} - e_{\boldsymbol{\hat k}(n)} + N_i^W(\boldsymbol{\hat k}(n))}{N_i^W(\boldsymbol{\hat k}(n))} \tag{since $e_{\boldsymbol{\hat k}(n)} \leq n \leq e_{\boldsymbol{\hat k}(n)+1}$}\\
    &= 1 + \frac{e_{\boldsymbol{\hat k}(n)+1} - e_{\boldsymbol{\hat k}(n)}}{N_i^W(\boldsymbol{\hat k}(n))}.
\end{align*}
By \cite[Page 295]{woess2009denumerable}, 
\[
\lim_{n \to \infty} \frac{e_{\boldsymbol{\hat k}(n)+1} - e_{\boldsymbol{\hat k}(n)}}{N_i^W(\boldsymbol{\hat k}(n))}=0, 
\]
almost surely, and so the first ratio is handled.

Turning to the third ratio of \eqref{fhfjcdc}, by similar arguments we have
\begin{align*}
    1 \le \frac{|Z_n|}{|W_{k(n)}|}
    \le 1 + \frac{e_{\boldsymbol{\hat k}(n)+1} - e_{\boldsymbol{\hat k}(n)}}{|W_{k(n)}|},
\end{align*}
and by \cite[Page 295]{woess2009denumerable}, the second term $\frac{e_{\boldsymbol{\hat k}(n)+1} - e_{\boldsymbol{\hat k}(n)}}{|W_{k(n)}|}$ tends to $0$ almost surely. This completes the proof.
\end{proof}

\section{Asymptotic growth rate}
\label{Asymptotic Growth Rate}

In this section, we are interested in the asymptotic growth rate of the polymer in the transient regime. Specifically, we ask whether there exists a constant $v \in (0, \infty)$ such that
\begin{align*}
    \lim_{t \to \infty} \frac{|X(t)|}{t} = v, \quad \text{almost surely},
\end{align*}
and we want to characterize the value of $v$.
We will prove the following.
\begin{theorem}
\label{thm:velocity}
Let $\bar{\sigma}_r$ denote the limiting proportion of monomers of type $M_r$, as given in Theorem~\ref{thm:proportion}. Then, in the transient regime, i.e, when $\alpha = \sum_{i=1}^d \frac{k_i^{+}}{k_i^{-}} > 1$, the process admits a deterministic asymptotic growth velocity $v \in (0,\infty)$ given by 
\begin{align*}
    v := \lim_{t \to \infty} \frac{|X(t)|}{t} 
    = \sum_{r=1}^{d} k_r^{+} - \sum_{r=1}^{d} k_r^{-} \,\bar{\sigma}_r,
    \quad \text{almost surely},
\end{align*}
where $\{k_r^{+}, k_r^{-}\}_{r=1}^d$ are the attachment and detachment rates of the respective monomer types.
\end{theorem}

Intuitively, the polymer’s growth velocity should reflect the net rate of monomer addition, weighted by how often the process occupies states ending in each monomer type. If we let
\[
A_r(t) \;=\; \frac{1}{t}\int_0^t \mathbf{1}_{\{\text{terminal monomer of } X(s) \text{ is } M_r\}}\,ds,
\]
then a natural heuristic is
\[
v \;=\; \sum_{r=1}^d k_r^+ \;-\; \sum_{r=1}^d k_r^- \, \lim_{t\to\infty} A_r(t),
\]
interpreting $\lim_{t\to\infty} A_r(t)$ as the long-time fraction of time that $X$ spends at polymers whose terminal monomer is $M_r$.

The cone-type and boundary-process analysis in Section~\ref{Limiting portion of each type of monomer} identifies $\bar{\sigma}_r$ as the limiting fraction of \emph{growth steps} at which the terminal monomer is $M_r$. It is therefore tempting to set $\lim_{t\to\infty} A_r(t) = \bar{\sigma}_r$ and conclude the formula for $v$ directly. However, this identification is not immediate: the boundary process records only the sequence of terminal monomers at successive growth events and ignores the random excursions of the continuous-time process $X$ between these events. In particular, the growth events do not occur at stopping times for $X$, so standard ergodic or renewal arguments cannot be applied without further work.

This section is devoted to bridging this gap. Rather than directly analyzing the time averages $A_r(t)$, we instead work directly with the process $|X(t)|$ and establish a strong law of large numbers for $\frac{|X(t)|}{t}$ in the transient regime. The resulting expression for the velocity coincides with the above heuristic formula, thereby providing a rigorous justification for interpreting $\bar{\sigma}_r$ as the effective contribution of each monomer type to the long-time growth rate.

Before proving Theorem \ref{thm:velocity}, we require some preliminary results.  It is most convenient to shift our analysis, as much as possible, to the DTMC $Z$.   Recall that $\tau_n$ is the time of the $n$-th jump of $X$, with $\tau_0 = 0$,  
and $Z_n = X(\tau_n)$ is the embedded DTMC.  
Define
\begin{align*}
    J_t := \max \left\{ n \in \mathbb{N} \,\middle|\, \tau_n \le t \right\},
\end{align*}
which represents the number of jumps of $X$ that have occurred at or before time $t$.  
Since $X_t = Z_{J_t}$, it follows that
\begin{align}
    \lim_{t \to \infty} \frac{t}{|X_t|} = \lim_{t \to \infty} \frac{t}{|Z_{J_t}|},
    \label{ddccvvbb}
\end{align}
if the limits exist.

We begin by getting useful upper and lower bounds on the numerator $t$. The process $X$ is a CTMC and so its holding times are exponentially distributed.   Since the $j$th state visited by the chain is $Z_j$, we may denote these holding times via $H_{Z_j}$.  We then note that $\tau_{n+1} =\sum_{j=0}^n H_{Z_j}$, and so define $H^n := \sum_{j=0}^n H_{Z_j}$ and
\begin{align*}
    &(H^n)_i := \sum_{j=0}^n H_{Z_j} \cdot \mathbf{1}_{C_i}\!\left(C(Z_j)\right), \quad i = 1,\dots,d,\\
    &(H^n)_\root := \sum_{j=0}^n H_{Z_j} \cdot \mathbf{1}_{\root}(Z_j).
\end{align*}
The above give (i)  the total amount of time the process has spent in states with various cone types up to time $\tau_{n+1}$ and (ii) the total amount of time the process has spent in the root.
Clearly,
\begin{align}
\label{fghfghfgh}
    H^n = (H^n)_\root+\sum_{i=1}^d (H^n)_i.
\end{align}
Moreover, $ H^{J_t - 1} \le t \le H^{J_t}$, and hence
\begin{align*}
    \frac{H^{J_t - 1}}{|Z_{J_t}|} 
    \;\le\; \frac{t}{|Z_{J_t}|} 
    \;\le\; \frac{H^{J_t}}{|Z_{J_t}|}.
\end{align*}
Applying the squeeze theorem, it therefore suffices to determine
\begin{align*}
    \lim_{t \to \infty} \frac{H^{J_t - 1}}{|Z_{J_t}|}  \quad  \text{ and } \quad \lim_{t \to \infty}\frac{H^{J_t}}{|Z_{J_t}|}.
\end{align*}
Note that because the process is transient, we have $\lim_{t \to \infty}  \frac{(H^{J_t-1})_\root}{|Z_{J_t}|} = 0$  and  $\lim_{t \to \infty}  \frac{(H^{J_t})_\root}{|Z_{J_t}|} = 0$ almost surely.  Hence, in view of \eqref{fghfghfgh}, our analysis  reduces to analyzing
\begin{align}
\label{sdjdfcjfcc}
    \lim_{t \to \infty}\frac{(H^{J_t - 1})_i}{|Z_{J_t}|} \quad \text{ and } \quad   \lim_{t \to \infty}\frac{(H^{J_t})_i}{|Z_{J_t}|}
\end{align}
for each $i \in \{1,\dots,d\}$.
The arguments for the two limits in \eqref{sdjdfcjfcc} are essentially the same and so we only focus on the second.    

We require one more bit of notation.  For each $i \in \{1, \dots, d\}$, we let
\begin{align}
\label{eq:cone_counts}
    \chi_i(n) := \sum_{j=0}^n \mathbf{1}_{C_i}\!\left(C(Z_j)\right),
\end{align}
be the number of visits to polymers with cone type $C_i$ in the first $n+1$ states of the process $Z$, and let $\chi_\root(n)$ be the number of visits to the root.
Since $J_t \to \infty$ almost surely as $t \to \infty$,  we have
\begin{align}
\label{eq:holding_time_limit}
    \lim_{t \to \infty} \frac{(H^{J_t})_i}{|Z_{J_t}|}
    = \lim_{n \to \infty} \frac{(H^n)_i}{|Z_n|}
    = \lim_{n \to \infty} \left( \frac{(H^n)_i}{\chi_i(n)} \cdot \frac{\chi_i(n)}{n} \cdot \frac{n}{|Z_n|} \right),
\end{align}
with probability one, so long as the limits exist.  Hence, it is sufficient to calculate the following three limits:
\begin{align*}
    \lim_{n\to \infty} \frac{(H^n)_i}{\chi_i(n)}, \qquad \lim_{n\to \infty} \frac{\chi_i(n)}{n}, \qquad \text{and}\qquad \lim_{n\to \infty} \frac{n}{|Z_n|}.
\end{align*}

The first of the above limits is straightforward.  From the previous section, we know that $\bar \sigma_i > 0$, and so $\chi_i(n) \to \infty$, as $n \to \infty$.  Hence, from the law of large numbers,
\begin{align}
\label{frfurfbf}
    \lim_{n\to \infty} \frac{(H^n)_i}{\chi_i(n)} = \frac{1}{k_i^{-} + \sum_{r=1}^d k_r^{+} }, \quad \text{almost surely}.
\end{align}
Moreover, the third limit is known, and we simply cite a result (see \cite[Theorem 9.100, Exercise 9.101]{woess2009denumerable}).

\begin{lemma}
\label{thm:velocity_discrete}
    The following limit holds with probability one,
    \begin{align*}
        \lim_{n \to \infty} \frac{|Z_n|}{n} = \bar{v}, \quad \text{where} \quad \bar{v} = \left( \sum_{i=1}^d \bar{\sigma}_i \cdot \frac{F_i}{\frac{k_i^{-}}{k_i^{-} + \sum_{r=1}^d k_r^{+} } (1 - F_i)} \right)^{-1}.
    \end{align*}
\end{lemma}

 For the middle term, $\lim_{n\to \infty} \frac{\chi_i(n)}{n}$, we have the following lemma.

 \begin{lemma}
 \label{lemma:89767685}
     The following limit holds with probability one,
     \begin{align}\label{li valueNEW}
    \lim_{k \rightarrow \infty} \frac{\chi_i(n)}{n} =\frac{\bar{\sigma}_i \left(k_i^- + \sum_{r=1}^d k_r^+  \right)}{\sum_{j=1}^d \bar{\sigma}_j \left(k_j^- + \sum_{r=1}^d k_r^+  \right)}.
\end{align}
\end{lemma}

The proof of Lemma \ref{lemma:89767685} is somewhat lengthy, so we postpone it until later. For now, we rely on it to establish Theorem \ref{thm:velocity}, the main result of this section.

\begin{proof}[Proof of Theorem \ref{thm:velocity}]
The proof essentially consists of plugging in the three pieces detailed above.  Noting that $\lim_{t\to\infty}\frac{(H^{J_t})_{\root}}{|Z_{J_t}|}=0$ almost surely, together with the three limits needed for \eqref{eq:holding_time_limit} above, we have that 
\begin{align*}
    \lim_{t \to \infty} \frac{H^{J_t}}{|Z_{J_t}|}
    = \lim_{t \to \infty} \sum_{i=1}^d \frac{(H^{J_t})_i}{|Z_{J_t}|} + \lim_{t \to \infty} \frac{(H^{J_t})_\root}{|Z_{J_t}|} 
    = \sum_{i=1}^d \frac{1}{k_i^- + \sum_{r=1}^d k_r^+ } \cdot \frac{\bar{\sigma}_i \left(k_i^- + \sum_{r=1}^d k_r^+  \right)}{\sum_{j=1}^d \bar{\sigma}_j \left(k_j^- + \sum_{r=1}^d k_r^+  \right)} \cdot \frac{1}{\bar{v}} ,
\end{align*}
with probability one, where according to Lemma \ref{thm:velocity_discrete}, 
    \begin{align*}
        \frac{1}{\bar{v}} = \sum_{a=1}^d \bar{\sigma}_a \cdot \frac{F_a}{\frac{k_a^{-}}{k_a^{-} + \sum_{r=1}^d k_r^{+} } (1 - F_a)} .
    \end{align*}
After some  algebra, we have the following almost sure limit 
    \begin{align*}
        &\lim_{t \to \infty} \frac{H^{J_t}}{|Z_{J_t}|} =\frac{1}{\sum_{r=1}^{d} k_r^{+}  - \sum_{r=1}^{d} \bar{\sigma}_r k_r^-}.
    \end{align*} 
By the same token, with probability one we also have
\begin{align*}
    \lim_{t \to \infty} \frac{H^{J_t - 1}}{|Z_{J_t}|}
    &= \sum_{i=1}^d \frac{1}{k_i^- + \sum_{r=1}^d k_r^+ } \cdot \frac{\bar{\sigma}_i \left(k_i^- + \sum_{r=1}^d k_r^+  \right)}{\sum_{j=1}^d \bar{\sigma}_j \left(k_j^- + \sum_{r=1}^d k_r^+  \right)}  \cdot \frac{1}{\bar{v}} = \frac{1}{\sum_{r=1}^d k_r^+  - \sum_{r=1}^d \bar{\sigma}_r k_r^-}.
\end{align*}
Recalling $\frac{H^{J_t - 1}}{|Z_{J_t}|} \leq \frac{t}{|Z_{J_t}|} \leq \frac{H^{J_t}}{|Z_{J_t}|}$, an application of  the squeeze theorem completes the proof.  
\end{proof}

With our main result in hand, the remainder of this section, and the appendix, is dedicated to proving Lemma \ref{lemma:89767685}.

\subsection{Proof of Lemma \ref{lemma:89767685}}
The proof of Lemma \ref{lemma:89767685} consists of two main steps.  
\begin{enumerate}
    \item We first establish the limit
    \begin{equation}
    \label{eq:980798078}
    \lim_{n\to \infty} \frac{\mathbb{E}_{\root}[\chi_i(n)]}{n} = \frac{\bar{\sigma}_i \left(k_i^- + \sum_{r=1}^d k_r^+  \right)}{\sum_{j=1}^d \bar{\sigma}_j \left(k_j^- + \sum_{r=1}^d k_r^+  \right)}.
    \end{equation}
    \item We then prove that the limit $\lim_{n\to \infty} \frac{\chi_i(n)}{n}$ exists almost surely.
\end{enumerate}

Assuming the above are established, the proof is straightforward.

\begin{proof}[Proof of Lemma \ref{lemma:89767685}]
    Since $\left| \frac{\chi_i(n)}{n} \right| \leq 1$, according to the Bounded Convergence Theorem, we have
\begin{align}\label{li value}
    \lim_{k \rightarrow \infty} \frac{\chi_i(n)}{n} = \lim_{k \rightarrow \infty} \frac{\E_{\root}[\chi_i(n)]}{n} =\frac{\bar{\sigma}_i \left(k_i^- + \sum_{r=1}^d k_r^+  \right)}{\sum_{j=1}^d \bar{\sigma}_j \left(k_j^- + \sum_{r=1}^d k_r^+  \right)}.
\end{align}
where the last equality comes from \eqref{eq:980798078}. 
 \end{proof}

We break the analysis of the two points above into two separate subsections.

\subsubsection{Proof of the limit \eqref{eq:980798078}}

For any two states $x, y \in \mT$, we define the function
\begin{align}
\label{G definition}
G(x, y) :=  \sum_{n=0}^{\infty} p^{(n)}(x, y),
\end{align}
where $p^{(n)}(x, y) = P_x(Z_n = y)$.  Note that $G(x,y)$ gives the total expected number of visits to state $y$ given an initial condition of $x$.  From \cite[Section 2]{Nagnibeda2002RWOT}, we know  $G(x, y) < \infty$ for any $x, y \in \mT$, and we also know the following lemma holds.

\begin{lemma}\cite[Lemma 2.1]{Nagnibeda2002RWOT}
\label{G properties}
Let $F$ be the  function defined in \eqref{Sec4:F}. Then the following relations hold: for any $x \in \mT$ and any monomer type $M_i \in \mathcal{M}$,
\begin{align}\label{relationship of G and F}
\begin{split}
    &F(x M_i, \root) = F(x M_i, x) \cdot F(x, \root) \quad \text{ if } x \neq \root,   \\
    &G(x, x)= \frac{1}{1 - F(x,x)},\\
    &G(x, y) = F(x, y) \cdot G(y, y) \quad \text{if } x \neq y.
\end{split}
\end{align}
\end{lemma}

We define  recursively the following (reversible) measure,
\begin{align}
\label{measure}
    \mu(\root) = 1, \quad \text{and} \quad \mu(x) = \mu(x^{-}) \cdot \frac{p(x^{-}, x)}{p(x, x^{-})} \quad \text{for } x \neq \root.
\end{align}

\begin{corollary}
\label{reversibility}
    For any $y \in \mT$,
    \begin{align*}
        G(\root, y) \cdot \mu(\root) = G(y, \root) \cdot \mu(y).
    \end{align*}
\end{corollary}

\begin{proof}
Let $\Gamma_{x,y}^{n}$ denote the set of all admissible paths of length $n$ from $x$ to $y$, that is,
\begin{align*}
    \Gamma_{x,y}^{n} := \left\{(z_0,z_1,\dots,z_n) \in \mT^{n+1} \ | z_0 = x,\ z_n = y,\ p(z_i, z_{i+1}) > 0, \ \forall\, i \in \{0, \dots, n-1\} \right\}.
\end{align*}
For any $y \in \mT$, the $n$-step transition probability $p^{(n)}(\root, y)$ can then be expressed as the sum over all such paths:
\begin{align*}
    p^{(n)}(\root, y) = \sum_{(z_0, z_1, \ldots, z_n) \in \Gamma_{\root,y}^{n}} p(z_0, z_1) \, p(z_1, z_2) \cdots p(z_{n-1}, z_n).
\end{align*}

Using the reversibility condition in \eqref{measure}, which relates $\mu$ and $p$ for each adjacent pair $(z_i, z_{i+1})$, we have for $(z_0, z_1, \ldots, z_n) \in \Gamma_{\root,y}^{n}$:
\begin{align*}
    \mu(\root) p(\root, z_1) &= \mu(z_1) p(z_1, \root), \\
    \mu(z_1) p(z_1, z_2) &= \mu(z_2) p(z_2, z_1),\\
    &\vdots\\
    \mu(z_{n-1}) p(z_{n-1},z_n) &= \mu(z_n)p(z_n,z_{n-1}).
\end{align*}
Applying this relation repeatedly yields
\begin{align}
\label{reversibility over paths}
\begin{split}
    &\mu(\root)\, p(\root, z_1)\, p(z_1, z_2)\, \cdots\, p(z_{n-1}, y) \\
    &= \mu(z_1)\, p(z_1, \root)\, p(z_1, z_2)\, \cdots\, p(z_{n-1}, y) \\
    &= \mu(z_2)\, p(z_2, z_1)\, p(z_1, \root)\, p(z_2, z_3)\, \cdots\, p(z_{n-1}, y) \\
    &\quad \vdots \\
    &= \mu(y)\, p(y, z_{n-1})\, p(z_{n-1}, z_{n-2})\, \cdots\, p(z_1, \root).
\end{split}
\end{align}
Summing over all such paths yields
\begin{align*}
    \mu(\root) p^{(n)}(\root,y) 
    &= \mu(\root) \sum_{(z_0,\dots,z_n)\in \Gamma_{\root,y}^{n}} p(z_0,z_1)\cdots p(z_{n-1},z_n) \\
    &= \mu(y) \sum_{(z_0,\dots,z_n)\in \Gamma_{y,\root}^{n}} p(z_0,z_1)\cdots p(z_{n-1},z_n) \\
    &= \mu(y) p^{(n)}(y,\root),
\end{align*}
where the second equality uses \eqref{reversibility over paths}, the third reindexes the reversed paths.
Finally, summing over $n \geq 0$, we obtain
\begin{align*}
        G(\root, y) \cdot \mu(\root) 
    = \sum_{n=0}^\infty p^{(n)}(\root, y)\, \mu(\root) 
    = \sum_{n=0}^\infty p^{(n)}(y, \root)\, \mu(y) = G(y, \root) \cdot \mu(y),
\end{align*}
and the result is shown.
\end{proof}

For each $i,j \in \{1,\dots, d\}$, we now compute the ratio 
$\frac{G(\root, x M_i)}{G(\root, x M_j)}$, which will play an important role  later.

\begin{proposition}
For any $x \in \mT$ and any monomer types $M_i, M_j \in \mathcal{M}$,
\begin{align}
\label{ratio_d_types}
    R_{ij} := \frac{G\left(\root, x M_j\right)}{G\left(\root, x M_i\right)} 
    = \frac{\bar \sigma_j\left(k_j^{-}+\sum_{r=1}^{d} k_r^{+}\right)}
           {\bar \sigma_i\left(k_i^{-}+\sum_{r=1}^{d} k_r^{+}\right)}.
\end{align}
\end{proposition}

\begin{proof}

By Corollary~\ref{reversibility}, for $x = \root$ we have
\begin{align*}
    \frac{G(\root, M_i)}{G(\root, \root)}
    &= \frac{\mu(\root) G(\root,  M_i)}{\mu(\root) G(\root, \root)} 
       = \frac{\mu( M_i) G( M_i,\root)}{\mu(\root) G(\root,\root)}.
\end{align*}
From \eqref{measure} and Lemma~\ref{G properties}, we have
\begin{align*}
    &\frac{\mu( M_i)}{\mu(\root)} = \frac{p(\root,  M_i)}{p( M_i, \root)}, \quad \text{and} \quad  G( M_i,\root) = F( M_i,\root) G(\root,\root),
\end{align*}
respectively.
Hence,
\begin{align*}
    \frac{G(\root,  M_i)}{G(\root, \root)}
    &= \frac{p(\root,  M_i) \, F( M_i, \root)G(\root,\root)}{p( M_i, \root)G(\root,\root)} \\
    &= \frac{p(\root,  M_i) \, F_i}{\left(\frac{k_i^{-}}{k_i^{-} + \sum_{r=1}^{d} k_r^{+}}\right)}.
\end{align*}
By the same argument,
\begin{align*}
    \frac{G(\root,  M_j)}{G(\root, \root)}
    = \frac{p(\root,  M_j) \, F_j}{\left(\frac{k_j^{-}}{k_j^{-} + \sum_{r=1}^{d} k_r^{+}}\right)}.
\end{align*}
Moreover, for each $x \in \mT$, $x \ne \root$ and $M_i, M_j \in \mathcal{M}$, we have
\begin{align*}
    \frac{G(\root, x M_i)}{G(\root, x)}
    &= \frac{\mu(\root) G(\root, x M_i)}{\mu(\root) G(\root, x)} 
       = \frac{\mu(x M_i) G(x M_i,\root)}{\mu(x) G(x,\root)},
\end{align*}
by Corollary~\ref{reversibility}.
From \eqref{measure} and  Lemma~\ref{G properties}, respectively, we have
\begin{align*}
    \frac{\mu(x M_i)}{\mu(x)} &= \frac{p(x, x M_i)}{p(x M_i, x)}, \\
    G(x M_i,\root) &= F(x M_i,\root) \, G(\root,\root) = F(x M_i,x) \, F(x,\root) \, G(\root,\root), \\
    G(x,\root) &= F(x,\root) \, G(\root,\root).
\end{align*}
Combining the above yields,
\begin{align*}
    \frac{G(\root, x M_i)}{G(\root, x)}
    &= \frac{p(x, x M_i) \, F(x M_i, x) \, F(x,\root) \, G(\root, \root)}{p(x M_i, x) \, F(x, \root) \, G(\root, \root)} 
    = \frac{p(x, x M_i) \, F(x M_i, x)}{p(x M_i, x)} \\
    &= \frac{p(x, x M_i) \, F_i}{\left(\frac{k_i^{-}}{k_i^{-} + \sum_{r=1}^{d} k_r^{+}}\right)}.
\end{align*}
By the same argument,
\begin{align*}
    \frac{G(\root, x M_j)}{G(\root, x)}
    = \frac{p(x, x M_j) \, F_j}{\left(\frac{k_j^{-}}{k_j^{-} + \sum_{r=1}^{d} k_r^{+}}\right)}.
\end{align*}
Finally, because
\begin{align*}
    \frac{p(x, x M_i)}{p(x, x M_j)} = \frac{k_i^+}{k_j^+},
\end{align*}
for all $x \in \mT$, 
we obtain
\begin{align*}
    \frac{G\left(\root, x M_j\right)}{G\left(\root, x M_i\right)} 
    &= \frac{p(x, x M_j) F_j}{\frac{k_j^{-}}{k_j^{-} + \sum_{r=1}^{d} k_r^{+}}}
       \cdot \frac{\frac{k_i^{-}}{k_i^{-} + \sum_{r=1}^{d} k_r^{+}}}
                  {p(x, x M_i) F_i} 
    = \frac{k_j^{+} k_i^{-} F_j \left(k_j^{-} + \sum_{r=1}^{d} k_r^{+}\right)}
            {k_i^{+} k_j^{-} F_i \left(k_i^{-} + \sum_{r=1}^{d} k_r^{+}\right)}
       = \frac{\bar \sigma_j \left(k_j^{-} + \sum_{r=1}^{d} k_r^{+}\right)}
              {\bar \sigma_i \left(k_i^{-} + \sum_{r=1}^{d} k_r^{+}\right)},
\end{align*}
where the last equality follows from \eqref{stationary dist}.
\end{proof}

We are now in a position to give the following proposition to compute the limit
\begin{align*}
    \lim_{n \to \infty} \frac{\mathbb{E}_{\root}[\chi_i(n)]}{n}.
\end{align*}

\begin{proposition}
\label{prop:ni_ratio_general}
For each $i \in \{1, 2, \dots, d\}$, the following limit holds: 
\begin{align}
\label{1/1+R}
    \lim_{n \to \infty} \frac{\mathbb{E}_{\root}[\chi_i(n)]}{n} 
    =  \lim_{n \to \infty} 
       \frac{\mathbb{E}_{\root}[\chi_i(n)]}
            {\sum_{j=1}^d \mathbb{E}_{\root}[\chi_j(n)]}
    = \frac{1}{ \sum_{j =1}^d  R_{ij}}
    = \frac{\bar{\sigma}_i \left(k_i^- + \sum_{r=1}^d k_r^+  \right)}
           {\sum_{j=1}^d \bar{\sigma}_j \left( k_j^- + \sum_{r=1}^d k_r^+  \right)},
\end{align}
where $\bar{\sigma}_i$ is the limiting proportion of monomer $M_i$ as given in Theorem~\ref{thm:proportion} and $R_{ij}$ is the ratio defined in \eqref{ratio_d_types}.
\end{proposition}

\begin{proof}
From the definition of $\{\chi_j(n)\}_{j \in \{1,2,\dots,d\}}$ in \eqref{eq:cone_counts}, we have
\begin{align*}
    \sum_{j=1}^d \chi_j(n) + \chi_\root(n) = n + 1.
\end{align*}
From \cite[Section 2]{Nagnibeda2002RWOT},
\begin{align*}
    \lim_{n\to \infty} \frac{\mathbb{E}_{\root}[\chi_\root(n)]}{n} = 0.
\end{align*}
Consequently, 
\begin{align*}
    \lim_{n \to \infty} \sum_{j=1}^d \frac{\mathbb{E}_{\root}[\chi_j(n)]}{n} 
    = \lim_{n \to \infty} \frac{n+1-\mathbb{E}_{\root}[\chi_\root(n)]}{n} = 1,
\end{align*}
and hence 
\begin{align*}
     \lim_{n \to \infty} \frac{\mathbb{E}_{\root}[\chi_i(n)]}{n}
     = \lim_{n \to \infty} \frac{\mathbb{E}_{\root}[\chi_i(n)]}{\sum_{j=1}^d \mathbb{E}_{\root}[\chi_j(n)]}.
\end{align*}
This establishes the first equality in \eqref{1/1+R} and the last equality comes from \eqref{ratio_d_types}. It remains to prove the second equality.

According to \eqref{ratio_d_types}, we obtain for each $i,j \in \{1,\dots,d\}$, 
\begin{align*}
    R_{ij} 
    &= 
       \frac{\sum_{x \in \mT} \frac{G(\root, x M_j)}{G(\root, x M_i)} \cdot G(\root, x M_i)}
            {\sum_{x \in \mT} G(\root, x M_i)} = 
       \frac{\sum_{x \in \mT} G(\root, x M_j)}
            {\sum_{x \in \mT} G(\root, x M_i)} =
      \frac{\sum_{x \in \mT} \sum_{\ell=0}^{\infty} p^{(\ell)}(\root, x M_j)}
            {\sum_{x \in \mT} \sum_{\ell=0}^{\infty} p^{(\ell)}(\root, x M_i)} .
\end{align*}
From \eqref{eq:cone_counts}, we have for any $i,j \in \{1,\dots,d\}$,
\begin{align*}
    &\lim_{n \to \infty} \frac{\mathbb{E}_{\root}[\chi_j(n)]}{\mathbb{E}_{\root}[\chi_i(n)]}
    = \lim_{n \to \infty} 
       \frac{\sum_{x \in \mT} \sum_{\ell=0}^{n} p^{(\ell)}(\root, x M_j)}
            {\sum_{x \in \mT} \sum_{\ell=0}^{n} p^{(\ell)}(\root, x M_i)} \\
    &= \lim_{n \to \infty} 
       \frac{\sum_{x \in \mT} \sum_{\ell=0}^{n} p^{(\ell)}(\root, x M_j)}
            {\sum_{x \in \mT} \sum_{\ell=0}^{\infty} p^{(\ell)}(\root, x M_j)} \cdot 
       \frac{\sum_{x \in \mT} \sum_{\ell=0}^{\infty} p^{(\ell)}(\root, x M_j)}
            {\sum_{x \in \mT} \sum_{\ell=0}^{\infty} p^{(\ell)}(\root, x M_i)} \cdot
       \frac{\sum_{x \in \mT} \sum_{\ell=0}^{\infty} p^{(\ell)}(\root, x M_i)}
            {\sum_{x \in \mT} \sum_{\ell=0}^{n} p^{(\ell)}(\root, x M_i)} .
\end{align*}
For the first term and last term, we have the following from the monotone convergence theorem:
\begin{align*}
     &\lim_{n \to \infty} \frac{\sum_{x \in \mT} \sum_{\ell=0}^{n} p^{(\ell)}(\root, x M_j)}
            {\sum_{x \in \mT} \sum_{\ell=0}^{\infty} p^{(\ell)}(\root, x M_j)}  =1 \qquad \text{and}\qquad
     \lim_{n \to \infty} \frac{\sum_{x \in \mT} \sum_{\ell=0}^{\infty} p^{(\ell)}(\root, x M_i)}
            {\sum_{x \in \mT} \sum_{\ell=0}^{n} p^{(\ell)}(\root, x M_i)}  =1.
\end{align*}
Recognizing that the middle term is simply $R_{ij}$,  we have
\begin{align*}
    \lim_{n \to \infty} \frac{\mathbb{E}_{\root}[\chi_j(n)]}{\mathbb{E}_{\root}[\chi_i(n)]}
    = R_{ij}.
\end{align*}
Finally, 
\begin{align*}
    \lim_{n \to \infty} \frac{\mathbb{E}_{\root}[\chi_i(n)]}{\sum_{j=1}^d \mathbb{E}_{\root}[\chi_j(n)]}
    = \lim_{n \to \infty} \frac{1}{\sum_{j =1}^d \frac{\mathbb{E}_{\root}[\chi_j(n)]}{\mathbb{E}_{\root}[\chi_i(n)]}}
    = \frac{1}{ \sum_{j =1}^d R_{ij}},
\end{align*}
which completes the proof.
\end{proof}

\subsubsection{Proof that $\lim_{n\to \infty} \frac{\chi_i(n)}{n}$ exists almost surely}
With the value $\lim_{n \to \infty} \frac{\mathbb{E}_{\root}[\chi_i(n)]}{n} 
    = \frac{1}{ \sum_{j =1}^d  R_{ij}}$ in hand, it remains to show that $\lim_{n\to \infty} \frac{\chi_i(n)}{n}$ exists for each $i\in \{1,\dots, d\}$. To that end, we now introduce the following notation for $x, y \in \mT$, $x \neq y$:
\begin{align}\label{S(x,y)}
\begin{split}
    &S_i^{(n)}(x, y) = \sum_{s=n}^{\infty} P_x\left( Z_s = y,\, Z_l \neq y \ \text{for } 0 \le l < s,\, \sum_{r=0}^{s-1} \mathbf{1}_{\left(C(Z_r) = C_i\right)} = n \right),\\
    &S_i(x, y) = \sum_{n=0}^{\infty} S_i^{(n)}(x, y), \\
    &S_i^{\prime}(x, y) = \sum_{n=0}^{\infty} n\, S_i^{(n)}(x, y).
\end{split}
\end{align}

\begin{remark}  
$S_i^{(n)}(x, y)$ denotes the probability that the process $Z$ visits $y$ for the \emph{first} time after at least $n$ steps, having visited polymers ending with $M_i$ exactly $n$ times before arriving at $y$.
\end{remark}

Similarly, we define
\begin{align}\label{T(x,y)}
\begin{split}
    &T_i^{(n)}(x, y) = \sum_{s=n}^{\infty} P_x\left( Z_s = y,\, \sum_{r=1}^{s} \mathbf{1}_{\left(C(Z_r) = C_i\right)} = n \right),\\
    &T_i(x, y) = \sum_{n=0}^{\infty} T_i^{(n)}(x, y).
\end{split}
\end{align}

\begin{remark} 
$T_i^{(n)}(x, y)$ is the probability that the process $Z$ reaches $y$ after at least $n$ steps, having visited polymers ending with $M_i$ exactly $n$ times \emph{before or at} $Z_s=y$ (excluding the starting state $x$).
\end{remark}

With these definitions, we can now observe the following relationships:
\begin{proposition}
\label{prop:768768}
For any $x, y \in \mT$, $x \neq y$, $S_i(x, y)=F(x, y)< 1$ \ for all $i \in \{1, \dots, d\}$.
\end{proposition}

\begin{proof} 
Let $x,y\in \mT$ with $x\ne y$.  Since $Z$ is transient, and because our state space is a tree, we have $F(x,y)<1$. Moreover, a straightforward calculation yields
\begin{align*}
    F(x,y) &= P_x(\text{$Z$ eventually hits $y$ in finite time})= \sum_{s=0}^\infty P_x(Z_s = y, Z_l \ne y \text{ for } 0 \leq l < s)\\
    &= \sum_{s=0}^\infty \sum_{n=0}^\infty P_x\left(Z_s = y, Z_l \ne y \text{ for } 0\leq l < s, \sum_{r=0}^{s-1} \mathbf{1}_{\left(C\left(Z_r\right)=C_i\right)}=n\right) \\
    &= \sum_{n=0}^\infty \sum_{s=0}^\infty P_x\left(Z_s = y, Z_l \ne y \text{ for } 0\leq l < s, \sum_{r=0}^{s-1} \mathbf{1}_{\left(C\left(Z_r\right)=C_i\right)}=n\right) \\
    &= \sum_{n=0}^\infty \sum_{s=n}^\infty P_x\left(Z_s = y, Z_l \ne y \text{ for } 0\leq l < s, \sum_{r=0}^{s-1} \mathbf{1}_{\left(C\left(Z_r\right)=C_i\right)}=n\right) \\
    &= \sum_{n=0}^\infty S_i^{(n)}(x,y) =S_i(x,y).
\end{align*}
\end{proof}

\begin{proposition}
For any $x, y \in \mT$, $x \neq y$, $T_i(x, y)=G(x, y)< \infty$ \ for all $i \in \{1, \dots, d\}$.
\end{proposition}
\begin{proof}
Since $Z$ is transient, $G(x,y)<\infty$ for any $x, y \in \mT$. By a similar argument as in the proof of Proposition \ref{prop:768768} above, for any $x, y \in \mT$, $i \in \{1, \dots, d\}$,
\begin{align*}
    G(x, y)&=\sum_{s=0}^{\infty} p^{(s)}(x, y)=\sum_{s=0}^{\infty} P_x\left(Z_s=y\right)
    =\sum_{n=0}^{\infty}\sum_{s=n}^{\infty} P_x\left(Z_s=y, \sum_{r=1}^s \mathbf{1}_{\left(C\left(Z_r\right)=C_i\right)}=n\right)\\
    &=\sum_{n=0}^{\infty}T_i^{(n)}(x, y) =T_i(x, y).
\end{align*}
\end{proof}

\begin{proposition}
\label{njfdvhjfdv}
    For  $x, y \in \mT$, with $x \neq y$, let  $F^{\prime}(x, y)=\sum_{n=0}^{\infty} n f^{(n)}(x, y)$.  Then, $S_i^{\prime}(x, y) \leq F'(x,y)$ \ for all $i \in \{1, \dots, d\}$.
\end{proposition}

\begin{proof}
    Let $x, y \in \mT$, with $x \neq y$.  For $i \in \{1, \dots, d\}$,
    \begin{align*}
        S_i^{\prime}(x, y) &=  \sum_{n=0}^{\infty} n S_i^{(n)}(x,y) = \sum_{n=0}^{\infty} n \sum_{s=n}^{\infty} P_x \left( Z_s = y, \ Z_l \neq y \text{ for } 0 \le l < s, \sum_{r=0}^{s-1} \mathbf{1}_{\left(C\left(Z_r\right)=C_i\right)}=n \right) \\
        &= \sum_{n=0}^{\infty}  \sum_{s=n}^{\infty} n P_x \left( Z_s = y, \ Z_l \neq y \text{ for } 0 \le l < s, \sum_{r=0}^{s-1} \mathbf{1}_{\left(C\left(Z_r\right)=C_i\right)}=n \right) \\
        &= \sum_{s=0}^{\infty} \sum_{n=0}^{s} n P_x \left( Z_s = y, \ Z_l \neq y \text{ for } 0 \le l < s, \sum_{r=0}^{s-1} \mathbf{1}_{\left(C\left(Z_r\right)=C_i\right)}=n \right) \\
        &\leq \sum_{s=0}^{\infty} \sum_{n=0}^{s} s P_x \left( Z_s = y, \ Z_l \neq y \text{ for } 0 \le l < s, \sum_{r=0}^{s-1} \mathbf{1}_{\left(C\left(Z_r\right)=C_i\right)}=n \right) \\
        &= \sum_{s=0}^{\infty} s \sum_{n=0}^{s}  P_x \left( Z_s = y, \ Z_l \neq y \text{ for } 0 \le l < s, \sum_{r=0}^{s-1} \mathbf{1}_{\left(C\left(Z_r\right)=C_i\right)}=n \right) \\
        &= \sum_{s=0}^{\infty} s P_x \left( Z_s = y, \ Z_l \neq y \text{ for } 0 \le l < s \right) = \sum_{s=0}^{\infty} s f^{(s)}(x,y) = F^{\prime}(x,y).\qedhere
    \end{align*}
\end{proof}

\begin{remark}
For any $x \in \mT \backslash \{\root\}$, the quantity $S_i^{(n)}(x, x^-)$  only depends upon the path of the process up to and including the first hitting time of $x^-$ when starting from $x$.  
By the tree structure of our process, all transitions  occur along the edges inside the subtree $T_x$.
Once the cone type of $T_x$ is known, the transition probabilities along these edges are fully determined, and hence the probability $S_i^{(n)}(x, x^-)$ is also determined.  
From this, we see that for each $j \in \{1,\dots,d\}$, all values $S_i^{(n)}(x M_j, x)$ are identical.  
Therefore, $S_i^{(n)}(x, x^-)$ depends only on  $n$, $i$ and cone type of $T_x$.  
In particular, for all $x \in \mT$ and for each $i, j \in \{1,\dots,d\}$, we have
\begin{align}\label{S_i_def_general}
\begin{split}
    &0<S_i^j := S_i(x M_j, x) = F(x M_j, x) = F_j < 1, \\
    &0<(S_i^j)^{\prime} := S_i^{\prime}(x M_j, x) \leq F_j^{\prime} :=F^{\prime}(x M_j, x) < \infty,
\end{split}
\end{align}
where the bound $F^{\prime}(x M_j, x) < \infty$ follows from \cite[Lemma 9.98]{woess2009denumerable}.
\end{remark}

With the quantities $S_i^{(n)}(x,y)$ defined in \eqref{S(x,y)}, we can now describe the transition mechanism of the process $\left(W_k,\chi_i(e_k)\right)_{k \geq 1}$.  This is one of our main results.  We note that this result is similar to Proposition 9.55 in \cite{woess2009denumerable} where it was shown that $(W_k, e_k)$ is a Markov chain. Here we are studying $\left(W_k,\chi_i(e_k)\right)$.  This difference is subtle but critical.

\begin{proposition}\label{transition prob. d}
    The process $\left(W_k,\chi_i(e_k)\right)_{k \geq 1}$ is a Markov chain for each $i \in \{1, \dots, d\}$. 
    In particular, for $x, y \in \mT$ with $|x| = k \geq 1$ and $y^{-} = x$ (so $|y| = k+1$), and $m, n \in \mathbb{Z}_{n \geq 0}$ with $n \geq m$, the transition probability is
    \begin{align*}
        P_\root&\left(W_{k+1} = y,\, \chi_i(e_{k+1}) = n \,\middle|\, W_k = x,\, \chi_i(e_k) = m\right) \\
        & = \frac{p(x, y)}{p\left(x, x^{-}\right)}
        \left( \frac{F\left(x, x^{-}\right)}{1 - F\left(x, x^{-}\right)} \right)
        \left( \frac{1 - F(y, x)}{F(y, x)} \right)
        S_i^{(n-m)}(y, x),
    \end{align*}
    where $S_i^{(n-m)}(y, x)$ is defined at \eqref{S(x,y)}:
    \begin{align*}
        S_i^{(n-m)}(y, x) 
        &= \sum_{s = n-m}^{\infty} 
        P_y\left(
            Z_s = x,\ 
            Z_l \neq x \ \text{ for } 0 \leq l < s,\ 
            \sum_{r = 0}^{s-1} \mathbf{1}_{\left(C(Z_r) = C_i\right)} = n - m
        \right).
    \end{align*}
\end{proposition}

We relegate the proof of Proposition \ref{transition prob. d} to Appendix \ref{AppendixProof1}.

Continuing, we denote the increment
\begin{align*}
    (\Delta_k)_i=\chi_i(e_{k})-\chi_i(e_{k-1}).
\end{align*}
We have the following proposition.
\begin{proposition}
    The process $\left(W_k, (\Delta_k)_i\right)_{k \geq 1}$ is a Markov chain.  In particular, for $x, y \in \mT$ with $|x|=k \geqslant 1$ and $y^{-}=x$ (so $|y| = k+1$), and $m, n \in \mathbb{Z}_{\geq 0}
$  , its transition probability is 
\begin{align*}
    &P_\root\left(W_{k+1}=y,\left(\Delta_{k+1}\right)_i=n \mid W_k=x,\left(\Delta_k\right)_i=m\right)\\
    &=\frac{p(x, y)}{p\left(x, x^{-}\right)}\left(\frac{F\left(x, x^{-}\right)}{1-F\left(x, x^{-}\right)}\right)\left(\frac{1-F(y, x)}{F(y, x)}\right) S_i^{(n)}(y, x).
\end{align*}
\end{proposition}

\begin{proof}
Let $x, y \in \mT$ with $|x| = k \geqslant 1$ and $y^{-} = x$.  
For all $m, n \in \mathbb{Z}_{\geq 0}$, we have
\begin{align*}
    P_\root(W_{k+1} = y,\, (\Delta_{k+1})_i = n & \mid W_k = x,\, (\Delta_k)_i = m) \\
    &= \frac{P_\root(W_{k+1} = y,\, (\Delta_{k+1})_i = n,\, W_k = x,\, (\Delta_k)_i = m)}
            {P_\root(W_k = x,\, (\Delta_k)_i = m)}.
\end{align*}
We treat the numerator and denominator separately.
We first handle the denominator $P_\root(W_k = x,\, (\Delta_k)_i = m)$.  
Since $(\Delta_k)_i = \chi_i(e_{k}) - \chi_i(e_{k-1})$, partitioning on $\chi_i(e_{k-1})$ yields
\begin{align*}
    P_\root(W_k = x,\, (\Delta_k)_i = m)
    = \sum_{l = 0}^{\infty} 
       P_\root(W_{k-1} = x^{-},\, \chi_i(e_{k-1}) = l,\, W_k = x,\, \chi_i(e_{k}) = l + m).
\end{align*}

Now consider the numerator
\begin{align*}
    P_\root(W_{k+1} = y,\, (\Delta_{k+1})_i = n,\, W_k = x,\, (\Delta_k)_i = m).
\end{align*}
Using $(\Delta_{k+1})_i = \chi_i(e_{k+1}) - \chi_i(e_{k})$ and partitioning on $\chi_i(e_{k-1})$ yields
\begin{align}
\label{EEE}
\begin{split}
    &P_\root(W_{k+1} = y, \  (\Delta_{k+1})_i = n, \  W_k = x, \ (\Delta_k)_i = m) \\
    &=P_\root (W_{k+1}=y, \ \chi_i(e_{k+1})-\chi_i(e_{k})=n, \ W_k=x, \  \chi_i(e_{k})-(e_{k-1})_i=m)\\
    &= \sum_{l = 0}^{\infty} 
       P_\root(W_{k-1} = x^{-}, \ \chi_i(e_{k-1}) = l, \ W_k = x, \ \chi_i(e_{k}) = l + m, \ W_{k+1} = y, \ \chi_i(e_{k+1}) = l + m + n).
\end{split}
\end{align}
Then we want to calculate the inner probability in \eqref{EEE}. First, we want to condition on the event
\begin{align*}
    \left\{W_{k-1}=x^{-}, \ \chi_i(e_{k-1})=l, \ W_k=x, \ \chi_i(e_k)=l+m \right\}
\end{align*}
and using Markov Property:
\begin{align*}
    &P_\root\left(W_{k-1}=x^{-}, \ \chi_i(e_{k-1})=l, \ W_k=x, \ \chi_i(e_k)=l+m, \ W_{k+1}=y, \ \chi_i(e_{k+1})=l+m+n \right)\\
    & \hspace{.2 in}=P_\root\left(W_{k+1}=y, \ \chi_i(e_{k+1})=l+m+n \mid W_{k-1}=x^{-},  \ \chi_i(e_{k-1})=l,  \ W_k=x ,  \ \chi_i(e_{k})=l+m \right) \\
    &  \hspace{.4 in} \times P_\root\left(W_{k-1}=x^{-}, \  \chi_i(e_{k-1})=l,  \ W_k=x, \ \chi_i(e_{k})=l+m \right)\\
    & \hspace{.2 in} =P_\root\left(W_{k+1}=y, \ \chi_i(e_{k+1})=l+m+n \mid W_k=x, \ \chi_i(e_{k})=l+m\right) \\
    &  \hspace{.4 in} \times P_\root\left(W_{k-1}=x^{-}, \  \chi_i(e_{k-1})=l, \ W_k=x, \ \chi_i(e_{k})=m \right).
\end{align*}
Turning back to \eqref{EEE}, we have
\begin{align*}
    &P_\root( W_{k+1} = y,\, (\Delta_{k+1})_i = n,\, W_k = x,\, (\Delta_k)_i = m ) \\
    &\hspace{.2 in}= \sum_{l = 0}^{\infty} 
       P_\root( W_{k-1} = x^{-},\ \chi_i(e_{k-1}) = l,\ W_k = x,\ \chi_i(e_{k}) = l + m,\ W_{k+1} = y,\ \chi_i(e_{k+1}) = l + m + n) \\
    &\hspace{.2 in}= \sum_{l = 0}^{\infty} 
       P_\root( W_{k+1} = y,\ \chi_i(e_{k+1}) = l + m + n \mid W_k = x,\ \chi_i(e_{k}) = l + m )\\
    &\hspace{.4 in}\times P_\root ( W_{k-1} = x^{-},\ \chi_i(e_{k-1}) = l, W_k = x,\ \chi_i(e_{k}) = l + m ) \\
    &\hspace{.2 in}= \frac{p(x, y)}{p\big(x, x^{-}\big)}
       \left( \frac{F\big(x, x^{-}\big)}{1 - F\big(x, x^{-}\big)} \right)
       \left( \frac{1 - F(y, x)}{F(y, x)} \right)
       S_i^{(n)}(y, x) \\
    &\hspace{.4 in}\times \sum_{l = 0}^{\infty} 
       P_\root( W_{k-1} = x^{-},\, \chi_i(e_{k-1}) = l,\, W_k = x,\, \chi_i(e_{k}) = l + m ),
\end{align*}
where the last equality follows from \eqref{TP}.

Collecting the above, we obtain
\begin{align}
\label{chchjdvvvbbv}
\begin{split}
        P_\root( W_{k+1} = &y, \ (\Delta_{k+1})_i = n \mid W_k = x, \ (\Delta_k)_i = m ) \\
    &= \frac{P_\root( W_{k+1} = y,\ (\Delta_{k+1})_i = n,\ W_k = x,\ (\Delta_k)_i = m )}
                 {P_\root ( W_k = x,\ (\Delta_k)_i = m )} \\
    &= \frac{p(x, y)}{p\big(x, x^{-}\big)}
       \left( \frac{F\big(x, x^{-}\big)}{1 - F\big(x, x^{-}\big)} \right)
       \left( \frac{1 - F(y, x)}{F(y, x)} \right)
       S_i^{(n)}(y, x).
\end{split}
\end{align}
\end{proof}

For the above transition probability \eqref{chchjdvvvbbv}, according to  \eqref{eq:F_i_sec4}, $\frac{F(x, x^{-})}{1 - F(x, x^{-})}$ and $\frac{1 - F(y, x)}{F(y, x)}$ depend on the cone types of $T_x$ and $T_y$, respectively. 
Moreover, $S^{(n)}_i(y, x)$ depends only on $n$, $i$, and the cone type of $T_y$ from \eqref{S_i_def_general}. Given all of the above, we conclude that the transition probability 
\begin{align*}
    P_\root(W_{k+1}=y,\ \left(\Delta_{k+1}\right)_i=n \mid W_k=x, \  (\Delta_k)_i=m)
\end{align*}
depends only on $n$, $i$, and the cone types of $T_x$ and $T_y$. 
Therefore, we can factorize with respect to the cone types, which implies that 
\begin{align*}
    \left(U_k,\left(\Delta_k\right)_i\right)_{k \geq 1}
\end{align*}
forms a Markov chain on $I \times \mathbb{Z}_{\ge 0}\backslash (C_i,0)$. The transition probability for $\left(U_k,\left(\Delta_k\right)_i\right)_{k \geq 1} $ is, for any $a,b \in \{1,2,\cdots,d\}$,
\begin{align*}
    \tilde{\mathrm{q}}((C_a, m),(C_b, n)) :&=P_\root(U_{k+1}=C_b, (\Delta_{k+1})_i=n \mid U_k=C_a,(\Delta_k)_i=m)\\
    &=\frac{k_b^{+}}{k_a^{-}} \cdot \frac{F_a}{1-F_a} \cdot \left( \frac{1-F_b}{F_b} \right) \cdot S_i^{b,n}\\
    &=\frac{k_b^{+}}{k_a^{-}} \cdot \frac{F_a}{1-F_a} \cdot (1-F_b) \cdot \frac{S_i^{b,n}}{S_i^b}\\
    &= V_{ab} \cdot \frac{S_i^{b,n}}{S_i^b} >0,
\end{align*}
where we define $S_i^{b,n} : = S_i^{(n)}(xM_b,x), \ x \in \mT$ and use $S_i^b = F_b$ from \eqref{S_i_def_general} in the last equality and 
\begin{align*}
    \{V_{ab}\}_{a,b \in \{1,2,\cdots,d\}}
\end{align*}
are the transition probabilities of the Markov chain $\left(U_k\right)_{k \geq 1}$, as defined in \eqref{general_M_matrix}.

With these probabilities in hand, we obtain the following proposition.
\begin{proposition}
\label{prop:lkdjafj}
    Fix $i\in\{1,\dots,d\}$. The bi-variate process $\left(U_k, (\Delta_k)_i\right)_{k \geq 1}$ is a positive recurrent Markov chain on $I \times \mathbb{Z}_{\ge 0} \backslash (C_i,0)$. Its stationary probability measure $\tilde{\sigma}$ is given by
    \begin{align*}
        \tilde{\sigma}(C_a, n) = \bar{\sigma}_a \frac{S_i^{a,n}}{S_i^a}, \quad \text{for all } (C_a, n) \in I \times \mathbb{Z}_{\ge 0} \backslash (C_i,0),
    \end{align*}
    where $\bar{\sigma}_a$ denotes the limiting proportion of cone type $C_a$ (equivalently, the limiting fraction of monomer $M_a$, as characterized in Theorem~\ref{thm:proportion}).
\end{proposition}

\begin{proof}

We see that $S_i^{b,n} > 0$, $V_{ab}>0$ for any $(C_a, m),(C_b, n) \in I \times \mathbb{Z}_{\ge 0} \backslash (C_i,0)$, then $\tilde{\mathrm{q}}((C_a, m),(C_b, n)) > 0$ for any $(C_a, m), (C_b, n) \in  I \times \mathbb{Z}_{\ge 0} \backslash (C_i,0)$, so $\left(U_k, (\Delta_k)_i\right)_{k \geq 1}$ is irreducible. Also, it's straightforward that $\tilde{\sigma}$ is a stationary probability measure. Since $ S_i^{i,0}=0$,
\begin{itemize}
    \item $\tilde{\sigma}$ is a probability measure:
    \begin{align*}
        \sum_{(C_a, n) \in  I \times \mathbb{Z}_{\ge 0}\backslash (C_i,0)} \tilde{\sigma}(C_a, n)
        = \sum_{a=1}^d \sum_{n \in \mathbb{Z}_{\ge 0}} \bar{\sigma}_a \frac{S_i^{a,n}}{S_i^a} = \sum_{a=1}^d \bar{\sigma}_a =1.
    \end{align*}
    \item $\tilde{\sigma}$ is a stationary probability measure: for any $(C_b, n) \in I \times \mathbb{Z}_{\ge 0}\backslash (C_i,0)$,
    \begin{align*}
        \sum_{(C_a, m) \in  I \times \mathbb{Z}_{\ge 0}\backslash (C_i,0)} \tilde{\sigma}(C_a, m) \tilde{q}((C_a, m), (C_b, n)) &= \sum_{a=1}^d  \sum_{m \in \mathbb{Z}_{\ge 0}} \bar{\sigma}_a \frac{S_i^{a,m}}{S_i^a} \cdot V_{ab} \cdot \frac{S_i^{b,n}}{S_i^b} \\
        &=\sum_{a=1}^d  \bar{\sigma}_a  V_{ab}  \frac{S_i^{b,n}}{S_i^b} \sum_{m \in \mathbb{Z}_{\ge 0}} \frac{S_i^{a,m}}{S_i^a} \\
        &=\bar{\sigma}_b \frac{S_i^{b,n}}{S_i^b} =\tilde{\sigma}(C_b, n) >0,
    \end{align*}
\end{itemize}
where we used that $\bar{\sigma}$ satisfies the stationary equation for the base cone-type Markov chain $\left(U_k\right)_{k \ge 1}$, i.e.,
\begin{align*}
    \bar{\sigma}_b = \sum_{a=1}^d \bar{\sigma}_a V_{ab}.
\end{align*}
Since there exists a positive stationary probability measure for $\left(U_k, (\Delta_k)_i\right)_{k \geq 1}$, $\left(U_k, (\Delta_k)_i\right)_{k \geq 1}$ is positive recurrent and the proof of Proposition \ref{prop:lkdjafj} is complete.
\end{proof}

Finally, we need the expectation of $(\Delta_k)_i$ under the stationary distribution just computed.  For that purpose, consider the projection $g_i: I \times \mathbb{Z}_{\ge 0}\backslash (C_i,0) \rightarrow \mathbb{Z}_{\ge 0},(a, n) \mapsto n$ for $\left(U_k, (\Delta_k)_i\right)_{k \geq 1}$. We have
\begin{align*}
    E_i: &= \int_{I \times \mathbb{Z}_{\ge 0}\backslash (C_i,0)} \  g_i \ d \tilde{\sigma} =\sum_{(C_a, n) \in  I \times \mathbb{Z}_{\ge 0}\backslash (C_i,0)} n \tilde{\sigma}(C_a, n) \\
    &=\sum_{(C_a, n) \in  I \times \mathbb{Z}_{\ge 0}\backslash (C_i,0)} n \bar{\sigma}_a \frac{S_i^{a,n}}{S_i^a}  = \sum_{a=1}^d \sum_{n=0}^\infty n \bar{\sigma}_a \frac{S_i^{a,n}}{S_i^a} \\
    &=\sum_{a=1}^d \bar{\sigma}_a \frac{(S_i^a)^{\prime}}{S_i^a} \\
    &\leq \sum_{a=1}^d \bar{\sigma}_a \frac{F_a^{\prime}}{F_a} \tag{by \eqref{S_i_def_general}}\\
    &< \infty.
\end{align*}
By the ergodic theorem for positive recurrent Markov chains, almost surely, for each $i\in\{1,\dots,d\}$,
\begin{align*}
    \lim_{k\to\infty}\frac{\chi_i(e_k)-\chi_i(e_0)}{k}
    = \lim_{k\to\infty}\frac{1}{k}\sum_{m=1}^k g_i\bigl(U_m,(\Delta_m)_i\bigr)
    = E_i < \infty.
\end{align*}
Moreover, as noted in \cite{woess2009denumerable}, 
\begin{align*}
    \lim_{k\to\infty}\frac{\chi_i(e_0)}{k}=0.
\end{align*}
Combining the previous two points yields the almost sure limit
\begin{equation}
\label{ergodic general d}
    \lim_{k\to\infty}\frac{\chi_i(e_k)}{k}=E_i<\infty.
\end{equation}
Note that this result pertains to the boundary process.  Hence, we must shift it to the nominal process $X$.

We recall the following integer-valued random variables
\begin{align*}
    \boldsymbol{\hat k}(n)=\max \left\{k: e_k \leq n\right\}.
\end{align*}
We have the following almost sure inequalities from \cite{woess2009denumerable}, 
\begin{align*}
    \boldsymbol{\hat k}(n) \to \infty, \quad \frac{e_{\boldsymbol{\hat k}(n)+1}}{e_{\boldsymbol{\hat k}(n)}} \to 1, \quad \frac{\chi_i \left(e_{\boldsymbol{\hat k}(n)+1} \right)}{\chi_i \left(e_{\boldsymbol{\hat k}(n)} \right) } \to 1,
\end{align*}
so that
\begin{align*}
    &0\leq\frac{n-e_{\boldsymbol{\hat k}(n)}}{n} \leq \frac{e_{\boldsymbol{\hat k}(n)+1}-e_{\boldsymbol{\hat k}(n)}}{n} \leq \frac{e_{\boldsymbol{\hat k}(n)+1}-e_{\boldsymbol{\hat k}(n)}}{e_{\boldsymbol{\hat k}(n)}} \rightarrow 0,\\
    &0\leq\frac{\chi_i(n)-\chi_i \left(e_{\boldsymbol{\hat k}(n)} \right)}{\chi_i(n)} \leq \frac{\chi_i \left(e_{\boldsymbol{\hat k}(n)+1} \right)-\chi_i \left(e_{\boldsymbol{\hat k}(n)} \right)}{\chi_i(n)} \leq \frac{\chi_i \left(e_{\boldsymbol{\hat k}(n)+1} \right)-\chi_i \left(e_{\boldsymbol{\hat k}(n)} \right)}{\chi_i \left(e_{\boldsymbol{\hat k}(n)} \right)} \to 0.
\end{align*}
Then for each $i = \{1, 2, \cdots, d\}$, we have as $n \to \infty$, almost surely,
\begin{align*}
    \frac{e_{\boldsymbol{\hat k}(n)}}{n} \to 1, \quad \frac{\chi_i \left(e_{\boldsymbol{\hat k}(n)} \right)}{\chi_i(n)} \to 1,
\end{align*}
so that for each $i = \{1, \dots, d\}$,
\begin{align*}
    \lim_{n \rightarrow \infty} \frac{\chi_i(n)}{n} 
    &= \lim_{n \rightarrow \infty} \frac{\chi_i(n)}{\chi_i \left(e_{\boldsymbol{\hat k}(n)} \right)} \cdot \frac{\chi_i \left(e_{\boldsymbol{\hat k}(n)} \right)}{e_{\boldsymbol{\hat k}(n)}} \cdot \frac{e_{\boldsymbol{\hat k}(n)}}{n} = \lim_{n \rightarrow \infty} \frac{\chi_i \left(e_{\boldsymbol{\hat k}(n)} \right)}{e_{\boldsymbol{\hat k}(n)}} \\
    &= \lim_{k \rightarrow \infty} \frac{\chi_i(e_{k})}{e_k} = \lim_{k \rightarrow \infty} \frac{\chi_i(e_{k})}{k} \cdot \frac{k}{e_k} =E_i \cdot \bar{v} < \infty,
\end{align*}
where the final equality follows from both  \eqref{ergodic general d} and \cite[Theorem~9.100]{woess2009denumerable}.
Hence, the existence of the limit has been verified.

\section{Copolymerization process involving two monomer types}
\label{Example: Copolymerization process involving two monomer types}

To illustrate our general results, we now specialize to the case of two monomer types. Thus, in this section we consider a copolymerization process $X$ involving the monomers $M_1 $ and $M_2$. The constants $k_i^{+}$ and  $k_i^{-}$ represent the attachment and detachment rates of monomer $M_i$, for  $i \in \{1,2\}$. The process can be visualized as in Figure \ref{fig:2monomerTree}. According to 
Theorem~\ref{thm:question1},  the recurrence/transience criterion for the copolymerization process \(X\) is given by the parameter
\begin{align*}
    \alpha \;=\; \frac{k_1^{+}}{k_1^{-}}+\frac{k_2^{+}}{k_2^{-}}.
\end{align*}
Specifically, $X$ is positive recurrent if \(\alpha<1\), null recurrent if \(\alpha=1\), and transient if \(\alpha>1\).

We begin by providing closed form solutions in this two-monomer case.  
Note that the results given here  are consistent with those presented in Section~III  of the paper ``Extracting chemical energy by growing disorder: Efficiency at maximum power'' \cite{esposito2010extracting}.

In the transient regime, the limiting proportion of each monomer type is given by
Theorem~\ref{thm:proportion} (see Section~\ref{Limiting portion of each type of monomer}). For the two–monomer case, when $X$ is transient, we obtain explicit formulas for the limiting proportions
\(\bar\sigma_1\) and \(\bar\sigma_2\) of \(M_1\) and \(M_2\), respectively.
We first consider the special case $k_1^{-} = k_2^{-}$. In this scenario, the almost-sure limiting proportions of $M_1$ and $M_2$ are
\begin{align}
\label{sigma111222}
     \bar{\sigma}_1=\lim_{t\to \infty} \sigma_{1}(t)  = \frac{k_1^{+}}{k_1^{+}+k_2^{+}}, 
     \qquad 
     \bar{\sigma}_2 =\lim_{t\to \infty} \sigma_{2}(t)  = \frac{k_2^{+}}{k_1^{+}+k_2^{+}}.
\end{align}
In the case of $k_1^{-} \neq k_2^{-}$, we obtain the almost-sure limiting proportions of $M_1$ and $M_2$ as
\begin{align}
\label{sigma222111}
\begin{split}
        \bar{\sigma}_1=\lim_{t\to \infty} \sigma_{1}(t) &=  \frac{k_1^{+}+k_2^{+}+k_1^{-}-k_2^{-}-\sqrt{(k_1^{+}+k_2^{+}+k_1^{-}-k_2^{-})^2 + 4 k_1^{+} k_2^{-} - 4 k_1^{+} k_1^{-}}}{2(k_1^{-}-k_2^{-})}, \\
    \bar{\sigma}_2=\lim_{t\to \infty} \sigma_{2}(t) &=  \frac{k_1^{+}+k_2^{+}+k_2^{-}-k_1^{-}-\sqrt{(k_1^{+}+k_2^{+}+k_2^{-}-k_1^{-})^2 + 4 k_2^{+} k_1^{-} - 4 k_2^{+} k_2^{-}}}{2(k_2^{-}-k_1^{-})}.
\end{split}
\end{align}
By Theorem \ref{thm:velocity},  the asymptotic growth velocity for the two–monomer case is given by
\begin{align*}
    v=\lim_{t \to \infty} \frac{|X_t|}{t} = k_1^{+}+k_2^{+}-\bar{\sigma}_1 k_1^{-}-\bar{\sigma}_2 k_2^{-},
\end{align*}
where $\bar{\sigma}_1$ and $\bar{\sigma}_2$ given at \eqref{sigma111222} and \eqref{sigma222111}.

Next, we provide simulated results of the two-monomer process with the following parameters,
\begin{equation}
    \label{eq:L90878908}
k_1^+ = 1, \quad   k_1^- = 1.8, \quad  k_2^+ = 1.2,   \quad \text{and} \quad  k_2^- = 2.592.
\end{equation}
Note that for these parameters, we have $k_1^+<k_1^-$ and $k_2^+<k_2^-$, but  $\alpha = \frac{1}{1.8} + \frac{1.2}{2.592} \approx 1.0185>1$, which puts us in the transient regime even though the detachment rate for each monomer is higher than its attachment rate. 
Specifically, we will visualize the convergence to the limiting proportion of each monomer type (Theorem \ref{thm:proportion}) and the limiting velocity of the growth of the process (Theorem \ref{thm:velocity}). Then, we will provide a visualization of the boundary process, which we used analytically throughout this paper.

For the parameters given in  \eqref{eq:L90878908}, we have
\[
\bar{\sigma}_1 \approx 0.5436 \quad \text{and} \quad \bar{\sigma}_2 \approx 0.4564.
\]
See Figure \ref{fig:actual_ratios} for our simulated results demonstrating the convergence of the proportions to these values.
\begin{figure}
    \centering
    \begin{subfigure}[t]{0.45\textwidth}
        \centering
        \includegraphics[width=\textwidth]{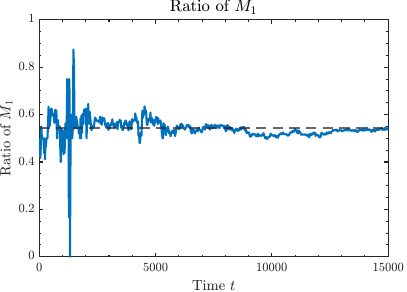}
        \caption{Empirical ratio of monomer $M_1$.}
        \label{fig:actual_ratio_1}
    \end{subfigure}
    \hfill
    \begin{subfigure}[t]{0.45\textwidth}
        \centering
        \includegraphics[width=\textwidth]{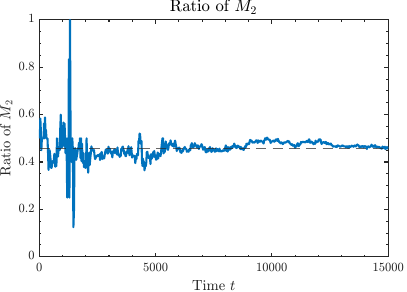}
        \caption{Empirical ratio of monomer $M_2$.}
        \label{fig:actual_ratio_2}
    \end{subfigure}
    \caption{
        Evolution of the empirical proportions of monomers $M_1$ and $M_2$ in the polymer with parameters given in \eqref{eq:L90878908}. The blue curves represent simulation results, while the gray dashed lines indicate the theoretical limiting values: $\bar{\sigma}_1 \approx 0.5436$ (left) and $\bar{\sigma}_2 \approx 0.4564$ (right).        
        } 
    \label{fig:actual_ratios}
\end{figure}
\begin{figure}
    \centering
    \includegraphics[width=0.5\linewidth]{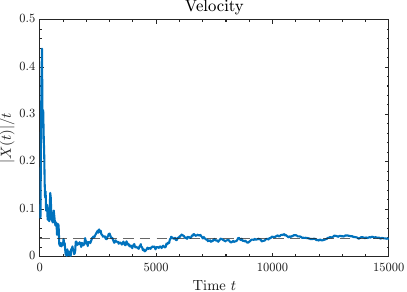}
    \caption{
        Empirical polymer growth velocity $\frac{|X_t|}{t}$ for the process with parameters \eqref{eq:L90878908}.  Note that the blue curve approaches the theoretical value $v \approx 0.0382$ (gray dashed line).
    }
    \label{fig:velocity}
\end{figure}
 In the simulation of the velocity $\frac{|X_t|}{t}$ in Figure \ref{fig:velocity} we see that the empirical velocity also stabilizes around the theoretical benchmark.

We turn to a visualization of the boundary process.   We remind the reader that this process played a critical role in the analysis carried out in sections \ref{Limiting portion of each type of monomer} and \ref{Asymptotic Growth Rate}.  In particular, we used that the boundary process remained ``close'' to the original process for all time (in a very specific manner), and we hope to demonstrate that visually here.  We recall that the boundary process was defined in order to keep track of the ``exit state'' at each level.  That is, $W_k$ was the particular state of our tree from level $k$ (i.e., was a polymer with $k$ monomers) that appears in the limiting ``infinite length'' polymer.  
That is, $W_k$ is   the unique prefix of the limiting polymer.

Note that the issue with simulating the boundary process is that the ``last exit time'' from a level is not a stopping time.  Since ``simulating to time infinity'' is not an option, we instead chose to simulate to a very large time, $T = 200,000$, and then restrict our visualization to a much smaller time-frame.  As before, we used the parameters \eqref{eq:L90878908}.   See Figures \ref{fig:boundary_vs_actual_1} and \ref{fig:boundary_vs_actual_2}, where close agreement between the actual process and the boundary process can be observed, especially as the time-frame increases.

\begin{figure}
    \centering
    \includegraphics[width=0.7\linewidth]{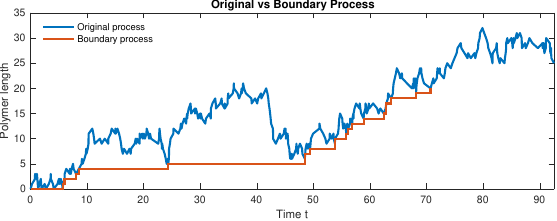}
    \caption{
        Comparison between the actual process and the boundary process in terms of polymer length over time period  $[0, 70]$.
    }
    \label{fig:boundary_vs_actual_1}
\end{figure}

\begin{figure}
    \centering
    \includegraphics[width=0.7\linewidth]{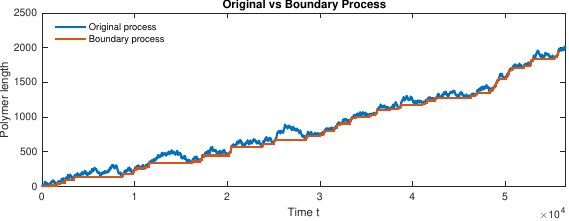}
    \caption{
        Comparison between the actual process and the boundary process in terms of polymer length over time period  $[0, 50,000]$.
    }
    \label{fig:boundary_vs_actual_2}
\end{figure}

\section{Discussion}
\label{Discussion}

Motivated by models from the Origins-of-Life literature, in this paper we studied a stochastic model of polymer growth. Earlier treatments focused  on two monomer types and relied on heuristic arguments. We provided rigorous analysis and, using this framework, extended the analysis seamlessly to the case of $d$ monomer types. The main contributions are as follows.

\begin{itemize}
    \item We formulated the copolymerization process with finitely many monomer types as a continuous-time Markov chain (CTMC) on an infinite, tree-like state space. By considering the embedded discrete-time Markov chain (DTMC), we characterized the positive recurrent, null recurrent, and transient conditions using spectral theory for random walks on trees with finitely many cone types.

    \item In the transient regime, we provide explicit formulas for the limiting monomer proportions $\{\bar{\sigma}_i\}$. These limits are characterized by the stationary distribution of an associated cone–type Markov chain obtained from the boundary process.

    \item We derived an explicit formula for the asymptotic velocity of polymer growth in the transient case. The expression involves the limiting monomer proportions $\bar{\sigma}_i$, and the transition rates, and relies on reversibility arguments.

\end{itemize}

Together, these results provide, to the best of our knowledge, the first mathematically rigorous treatment of this class of copolymerization models, generalizing and justifying earlier physics-based work. The methods introduced here—particularly the spectral criterion for transience, the cone-type Markov chain formalism, and the explicit construction of boundary processes—are broadly applicable to other stochastic models of assembly processes with hierarchical or rule-based structures.

Beyond the specific copolymerization setting, the analytic framework developed---particularly in Section~5 and particularly the use of boundary processes and the control of continuous-time excursions between successive growth levels---addresses a general class of problems for Markov chains on tree-like state spaces. In such models, natural macroscopic observables (such as growth rates or empirical frequencies) are often coupled to local transition mechanisms in a way that is not directly accessible through moment equations alone. The methods introduced here provide a systematic way to extract almost-sure asymptotic behavior in this setting, and we expect them to be useful in other stochastic growth processes with similar hierarchical structure.

To conclude, we mention possible avenues for future research. Most biochemically relevant processes can be modeled as rule-based systems \cite{danos2007rule}, in which a finite set of rules gives rise to an infinite cascade of assemblies and functions. These systems can sometimes be formally described using the double-pushout approach from category theory \cite{ehrig1973graph,andersen2016software}, and significant work in the computer science and mathematics communities has focused on formulating them as continuous-time Markov chains via generating function techniques \cite{behr2016stochastic, behr2021stochastic}, as well as developing methods for their simulation \cite{boutillier2018kappa}. More recently, algebraic approaches based on the Fock space formalism have been introduced in the physics literature to study rule-based systems \cite{rousseau2025algebraic}. The copolymerization process examined in this paper provides a simple yet illustrative example of a rule-based system. We hope that the rigorous treatment of the copolymerization model presented here will serve as a common platform for unifying different mathematical approaches to rule-based systems and for inspiring their extension to biologically significant problems.

\section*{Acknowledgments}

DFA gratefully acknowledges support from NSF grant DMS-2051498.
PG wants to acknowledge Eric Smith, Nicolas Behr, and Jean Krivine for technical discussions. 
PG was partially funded by the National Science Foundation, Division of Environmental Biology (Grant No: DEB-2218817), JST CREST (JPMJCR2011), and JSPS grant No: 25H01365.

\bibliographystyle{plain}
\bibliography{references}

\begin{thebibliography}{10}

\bibitem{andersen2016software}
Jakob~L. Andersen, Christoph Flamm, Daniel Merkle, and Peter~F. Stadler.
\newblock A software package for chemically inspired graph transformation.
\newblock In {\em International conference on graph transformation}, pages
  73--88. Springer, 2016.

\bibitem{andrieux2008nonequilibrium}
David Andrieux and Pierre Gaspard.
\newblock Nonequilibrium generation of information in copolymerization
  processes.
\newblock {\em Proceedings of the National Academy of Sciences},
  105(28):9516--9521, 2008.

\bibitem{behr2021stochastic}
Nicolas Behr.
\newblock On stochastic rewriting and combinatorics via rule-algebraic methods.
\newblock {\em arXiv preprint arXiv:2102.02364}, 2021.

\bibitem{behr2016stochastic}
Nicolas Behr, Vincent Danos, and Ilias Garnier.
\newblock Stochastic mechanics of graph rewriting.
\newblock In {\em Proceedings of the 31st Annual ACM/IEEE Symposium on Logic in
  Computer Science}, pages 46--55, 2016.

\bibitem{boutillier2018kappa}
Pierre Boutillier, Mutaamba Maasha, Xing Li, H{\'e}ctor~F Medina-Abarca, Jean
  Krivine, J{\'e}r{\^o}me Feret, Ioana Cristescu, Angus~G Forbes, and Walter
  Fontana.
\newblock The kappa platform for rule-based modeling.
\newblock {\em Bioinformatics}, 34(13):i583--i592, 2018.

\bibitem{danos2007rule}
Vincent Danos, J{\'e}r{\^o}me Feret, Walter Fontana, Russell Harmer, and Jean
  Krivine.
\newblock Rule-based modelling of cellular signalling.
\newblock In {\em International conference on concurrency theory}, pages
  17--41. Springer, 2007.

\bibitem{ehrig1973graph}
Hartmut Ehrig, Michael Pfender, and Hans~J{\"u}rgen Schneider.
\newblock Graph-grammars: An algebraic approach.
\newblock In {\em 14th Annual symposium on switching and automata theory (swat
  1973)}, pages 167--180. IEEE, 1973.

\bibitem{esposito2010extracting}
Massimiliano Esposito, Katja Lindenberg, and Christian Van~den Broeck.
\newblock Extracting chemical energy by growing disorder: efficiency at maximum
  power.
\newblock {\em Journal of Statistical Mechanics: Theory and Experiment},
  2010(01):P01008, 2010.

\bibitem{gagrani2025evolution}
Praful Gagrani and David Baum.
\newblock Evolution of complexity and the transition to biochemical life.
\newblock {\em Physical Review E}, 111(6):064403, 2025.

\bibitem{gaspard2016kinetics}
Pierre Gaspard.
\newblock Kinetics and thermodynamics of living copolymerization processes.
\newblock {\em Philosophical Transactions of the Royal Society A: Mathematical,
  Physical and Engineering Sciences}, 374(2080):20160147, 2016.

\bibitem{gaspard2020template}
Pierre Gaspard.
\newblock Template-directed growth of copolymers.
\newblock {\em Chaos: An Interdisciplinary Journal of Nonlinear Science},
  30(4), 2020.

\bibitem{gaspard2014kinetics}
Pierre Gaspard and David Andrieux.
\newblock Kinetics and thermodynamics of first-order markov chain
  copolymerization.
\newblock {\em The Journal of chemical physics}, 141(4), 2014.

\bibitem{koonin2009origin}
Eugene~V. Koonin and Artem~S. Novozhilov.
\newblock Origin and evolution of the genetic code: the universal enigma.
\newblock {\em IUBMB life}, 61(2):99--111, 2009.

\bibitem{Nagnibeda2002RWOT}
Tatiana Nagnibeda and Wolfgang Woess.
\newblock Random walks on trees with finitely many cone types.
\newblock {\em Journal of Theoretical Probability}, 15:383--422, 2002.

\bibitem{norris1998markov}
James~R. Norris.
\newblock {\em Markov chains}.
\newblock Number~2. Cambridge university press, 1998.

\bibitem{nowak2008prevolutionary}
Martin~A. Nowak and Hisashi Ohtsuki.
\newblock Prevolutionary dynamics and the origin of evolution.
\newblock {\em Proceedings of the National Academy of Sciences},
  105(39):14924--14927, 2008.

\bibitem{rousseau2025algebraic}
Rebecca~J. Rousseau and Justin~B. Kinney.
\newblock Algebraic and diagrammatic methods for the rule-based modeling of
  multiparticle complexes.
\newblock {\em PRX Life}, 3(2):023004, 2025.

\bibitem{woess2009denumerable}
Wolfgang Woess.
\newblock {\em Denumerable Markov chains}.
\newblock European Mathematical Society Z{\"u}rich, 2009.

\end{thebibliography}

\appendix

\section{Proof of Proposition \ref{transition prob. d}}
\label{AppendixProof1}

\begin{proof}[Proof of Proposition \ref{transition prob. d}]
Let $x, y \in \mT$ with $|x| = k \geq 1$ and $y^{-} = x$. Then for all $m, n \in \mathbb{Z}_{\geq 0}$ with $n \geq m$, we have
\begin{align}
\begin{split}\label{eq:8769876a}
    P_\root(W_{k+1} = y, \chi_i(e_{k+1}) = n& \mid W_k = x, \chi_i(e_{k}) = m)\\
    &= \frac{P_\root(W_{k+1} = y, \chi_i(e_{k+1}) = n, W_k = x, \chi_i(e_{k}) = m)}{P_\root(W_k = x, \chi_i(e_{k}) = m)}.
    \end{split}
\end{align}
We handle the numerator and denominator separately.  We begin with the denominator $P_\root(W_k=x, \chi_i(e_k)=m)$.
Partitioning on $e_k$ yields
\begin{align}\label{FFF_d}
    P_\root(W_k = x, \chi_i(e_{k}) = m)
    = \sum_{l = m}^{\infty} P_\root(W_k = x, e_k = l, \chi_i(e_{k}) = m).
\end{align}
For $l \geq m \geq 0$, $k \geq 1$, the conditional probability formula and the Markov property yield
\begin{align}
    P_\root(W_k = x, &e_k = l, \chi_i(e_{k}) = m)\notag \\
    &= P_\root\left( Z_l = x,\ Z_r \in \mT_x \setminus \{x\} \ \forall r \geq l + 1,\ \sum_{r = 0}^l \mathbf{1}_{(C(Z_r) = C_i)} = m \right)\notag \\
    &= P_\root\left(Z_r \in \mT_{x}\backslash \{x\} \ \forall r \ge l +1\ \big|\ Z_l = x,\ \sum_{r=0}^{l} \mathbf{1}_{(C(Z_r)=C_{i})} = m \right) \notag\\
    &\hspace{1em} \times P_\root\left(Z_l = x,\ \sum_{r=0}^{l} \mathbf{1}_{(C(Z_r)=C_{i})} = m\right)\notag\\
    &= P_x( Z_r \in \mT_x \setminus \{x\} \ \forall r \geq 1 ) \times P_\root \left( Z_l = x,\ \sum_{r = 0}^l \mathbf{1}_{(C(Z_r) = C_i)} = m \right).\label{QQQ_d}
\end{align}
From \cite[Equation 4.1]{Nagnibeda2002RWOT} we have:
\begin{align*}
    P_x(Z_r \in \mT_x \setminus \{x\} \ \forall r \geq 1)
    = \sum_{y: y^{-} = x} p(x, y) \, (1 - F(y, x)) = p(x, x^{-}) \left( \frac{1}{F(x, x^{-})} - 1 \right).
\end{align*}
Substituting this into \eqref{QQQ_d} gives
\begin{align*}
    P_\root(W_k = x, e_k = l, \chi_i(e_{k}) = m)
    = p(x, x^{-})\left( \frac{1}{F(x, x^{-})} - 1 \right)
      P_\root\left( Z_l = x,\ \sum_{r = 0}^l \mathbf{1}_{(C(Z_r) = C_i)} = m \right).
\end{align*}
Returning to \eqref{FFF_d}:
\begin{align*}
    P_\root(W_k = x,  \chi_i(e_{k}) = m) &= \sum_{l=m}^{\infty} P_\root\left(W_k=x, e_k=l, \chi_i(e_k)=m\right) \\
    &\quad = p(x, x^{-})\left( \frac{1}{F(x, x^{-})} - 1 \right)
      \sum_{l = m}^{\infty} P_\root \left( Z_l = x,\ \sum_{r = 0}^l \mathbf{1}_{(C(Z_r) = C_i)} = m \right) \\
    &\quad = p(x, x^{-})\left( \frac{1}{F(x, x^{-})} - 1 \right)
      \sum_{l = m}^{\infty} P_\root \left( Z_l = x,\ \sum_{r = 1}^l \mathbf{1}_{(C(Z_r) = C_i)} = m \right)  \\
    &\quad = p(x, x^{-})\left( \frac{1}{F(x, x^{-})} - 1 \right) T_i^{(m)}(\root, x),
\end{align*}
where $T_i^{(m)}(\root, x)$ is defined at \eqref{T(x,y)}.

Now we are ready to study the numerator of \eqref{eq:8769876a}.
We begin by partitioning the event over the possible values of $e_k$ and $e_{k+1} - e_k$:
\begin{align*}
    &P_\root(W_k = x, W_{k+1} = y,\chi_i(e_k) = m, \chi_i(e_{k+1}) = n) \\
    &= P_\root(W_k = x, W_{k+1} = y,\chi_i(e_k) = m,\chi_i(e_{k+1}) - \chi_i(e_k) = n - m) \\
    &= \sum_{l = m}^{\infty} \sum_{q = n - m}^{\infty} 
    P_\root(W_k = x, W_{k+1} = y, e_k = l, e_{k+1} - e_k = q, \chi_i(e_k) = m, \chi_i(e_{k+1}) - \chi_i(e_k) = n - m).
\end{align*}
Now we study the inner probability. For \( l \geq m \geq 0 \), \( q \geq n - m \geq 0 \), and \( k \geq 1 \), we have:
\begin{align*}
    &P_\root(W_k = x, W_{k+1} = y, e_k = l, e_{k+1} - e_k = q, \chi_i(e_k) = m, \chi_i(e_{k+1}) - \chi_i(e_k) = n - m) \\
    &= P_\root \left(
        \begin{aligned}
            &Z_l = x, \
            Z_{l + r} \in \mT_y \text{ for } r = 1, 2, \ldots, q - 1, \
            Z_{l + q} = y, \
            Z_{l + q + v} \in \mT_y \backslash \{y\} \ \forall\,v \geq 1, \\
            &\sum_{s = 0}^{l} \mathbf{1}_{(C(Z_s) = C_{i})} = m, 
            \sum_{s = l + 1}^{l + q} \mathbf{1}_{(C(Z_s) = C_{i})} = n - m
        \end{aligned}
    \right).
\end{align*}
Conditioning on the event
\begin{align*}
   \left\{Z_l = x,\ \sum_{s=0}^l \mathbf{1}_{\left(C(Z_s) = C_{i}\right)} = m\right\}
\end{align*}
and utilizing the Markov property yields
\begin{align*}
    &P_\root\left(W_k = x, W_{k+1} = y, e_k = l, e_{k+1} - e_k = q, \chi_i(e_k) = m,\ \chi_i(e_{k+1}) - \chi_i(e_k) = n - m\right) \\
    &= P_\root \left(
    \begin{aligned}
        &Z_{l+r} \in \mT_y \ \text{for } r = 1,\dots,q-1,\ 
        Z_{l+q} = y,\ 
        Z_{l+q+v} \in \mT_y \setminus \{y\} \ \forall\, v \geq 1, \\
        &\sum_{s=l+1}^{l+q} \mathbf{1}_{(C(Z_s) = C_{i})} = n - m \big| Z_l = x,\ \sum_{s=1}^l \mathbf{1}_{(C(Z_s) = C_{i})} = m
    \end{aligned}
    \right) \\
    &\hspace{.3in} \times P_\root\left(Z_l = x,\ \sum_{s=0}^l \mathbf{1}_{(C(Z_s) = C_{i})} = m\right) \\
    &= P_x\left(
       \begin{aligned}
        &Z_r \in \mT_y \ \text{for } r = 1,\dots,q-1,\ 
        Z_q = y,\ 
        Z_{q+v} \in \mT_y \setminus \{y\} \ \forall\, v \geq 1, \\
        &\sum_{s=1}^q \mathbf{1}_{(C(Z_s) = C_{i})} = n - m
       \end{aligned}
    \right) \\
    &\hspace{.3in} \times P_\root\left(Z_l = x,\ \sum_{s=1}^l \mathbf{1}_{(C(Z_s) = C_{i})} = m\right).
\end{align*}
Similarly, we can split out the event 
\begin{align*}
     \left\{Z_{q+v} \in \mT_y \setminus \{y\} \ \forall\, v \geq 1\right\}
\end{align*}
using the conditional probability formula and the Markov property:
\begin{align}\label{TTT}
\begin{split}
    &P_\root\left(W_k = x, W_{k+1} = y, e_k = l, e_{k+1} - e_k = q, \chi_i(e_k) = m,\ \chi_i(e_{k+1}) - \chi_i(e_k) = n - m\right) \\
    &= P_x \left( Z_{q+v} \in \mT_y \setminus \{y\} \ \forall\, v \geq 1 \ \Big|\
        Z_r \in \mT_y \ \text{for } r = 1,\dots,q-1,\ Z_q = y, \sum_{s=1}^q \mathbf{1}_{(C(Z_s) = C_{i})} = n - m
    \right) \\
    & \hspace{.3in} \times P_x \left(
        Z_r \in \mT_y \ \text{for } r = 1,\dots,q-1,\ Z_q = y,\
        \sum_{s=1}^q \mathbf{1}_{(C(Z_s) = C_{i})} = n - m
    \right)\\
    &\hspace{.3in}  \times P_\root\left(Z_l = x,\ \sum_{s=1}^l \mathbf{1}_{(C(Z_s) = C_{i})} = m\right) \\
    &= P_y\left(Z_v \in \mT_y \setminus \{y\} \ \forall\, v \geq 1\right) \\
    &\hspace{.3in} \times P_x\left(
        Z_r \in \mT_y \ \text{for } r = 1,\dots,q-1,\ Z_q = y,\
        \sum_{s=1}^q \mathbf{1}_{(C(Z_s) = C_{i})} = n - m
    \right) \\
    &\hspace{.3in}  \times P_\root\left(Z_l = x,\ \sum_{s=1}^l \mathbf{1}_{(C(Z_s) = C_{i})} = m\right).
\end{split}
\end{align}
From \cite[equation 4.1]{Nagnibeda2002RWOT}, we have for $W_k=x$ and $W_{k+1}=y$ (with $x=y^-$),
\begin{align}
\label{MMM}
    P_y\left(Z_v \in T_y \backslash\{y\} \ \forall\, v \geq 1\right) =\sum_{z: z^{-}=y} p(y, z)(1-F(z, y))= p\left(y, x\right)\left(\frac{1}{F(y, x)}-1\right).
\end{align}
We will postpone the analysis of the term 
\begin{align*}
    P_\root\left( Z_l = x, \ \sum_{s=1}^l \mathbf{1}_{\left(C(Z_s) = C_i\right)} = m \right)
\end{align*}
until later. For now, we focus on evaluating 
\begin{align*}
    P_x\left( Z_r \in \mT_y \ \text{for } r = 1, 2, \dots, q-1, \ Z_q = y, \ \sum_{s=1}^q \mathbf{1}_{\left(C(Z_s) = C_i\right)} = n - m \right).
\end{align*}
by making use of the reversibility property. From an argument analogous to that in \eqref{reversibility over paths}, for any path of length $q$ with $z_0 = x$, $z_q = y$, and $p(z_s, z_{s+1}) > 0$ for all $s \in \{0, \dots, q-1\}$, the reversibility condition gives
\begin{align}
\label{KKK}
    \mu(x) \, p\left(x, z_1\right) p\left(z_1, z_2\right) \cdots p\left(z_{q-1}, y\right)
    = \mu(y) \, p\left(y, z_{q-1}\right) \cdots p\left(z_2, z_1\right) p\left(z_1, x\right).
\end{align}
Specifically, consider the set of paths 
\begin{align*}
    \bar{\Gamma}_{x,y}^{q,\,n-m} 
    := \{ (z_0,z_1,\dots,z_q) \in \mT^{q+1} \}
\end{align*}
of length $q$ such that  
$z_0 = x$, $z_q = y$, $p(z_s, z_{s+1}) > 0$ for all $s \in \{0, \dots, q-1\}$, $z_s \in \mT_y$ for all $s \in \{1, \dots, q-1\}$ and
\begin{align*}
    \sum_{s=1}^{q} \mathbf{1}_{\left(C(z_s) = C_i\right)} = n - m.
\end{align*}
Then we have
\begin{align*}
    P_x\bigg( Z_r \in& \mT_y \ \text{for } r = 1, 2, \dots, q-1, \ Z_q = y,\ 
    \sum_{s=1}^{q} \mathbf{1}_{\left(C(Z_s) = C_i\right)} = n - m \bigg) \\
    &= \sum_{(z_0, z_1, \dots, z_q) \in \bar{\Gamma}_{x,y}^{q,\,n-m}} 
       p(x, z_1) \cdots p(z_{q-1}, y) \\
    &= \sum_{(z_0, z_1, \dots, z_q) \in \bar{\Gamma}_{x,y}^{q,\,n-m}} 
       \frac{\mu(y)}{\mu(x)} \, p(y, z_{q-1}) \cdots p(z_1, x)
       \tag{by \eqref{KKK}} \\
    &= \sum_{(z_0, z_1, \dots, z_q) \in \bar{\Gamma}_{x,y}^{q,\,n-m}} 
       \frac{p(x, y)}{p(y, x)} \, p(y, z_{q-1}) \cdots p(z_1, x)
       \tag{by \eqref{measure}} \\
    &= \frac{p(x, y)}{p(y, x)} \,
       P_y\left( Z_r \in \mT_y \ \text{for } r = 1, 2, \dots, q-1, \ Z_q = x,\ 
       \sum_{s=0}^{q-1} \mathbf{1}_{\left(C(Z_s) = C_i\right)} = n - m \right).
\end{align*}
We now sum over the length $q$ from $n-m$ to $\infty$ of the path $\bar{\Gamma}_{x,y}^{q,n-m}$,
\begin{align}\label{SSS}
\begin{split}
    &\sum_{q=n-m}^{\infty} P_{x}\left(Z_r \in T_y \text{ for } r=1,2, \cdots, q-1, Z_q=y,\sum_{s=1}^{q} \mathbf{1}_{\left(C\left(Z_s\right)=C_i\right)}=n-m \right)\\
    &=\frac{p\left(x, y\right)}{p\left(y, x\right)}\sum_{q=n-m}^{\infty} P_{y}\left(Z_r \in T_y \text{ for } r=1,2, \cdots, q-1, Z_q=x,\sum_{s=0}^{q-1} \mathbf{1}_{\left(C\left(Z_s\right)=C_i\right)}=n-m \right)\\
    &=\frac{p\left(x, y\right)} {p\left(y, x\right)}\sum_{q=n-m}^{\infty} P_y\left(Z_r \neq x \text{
    for }r=1,2, \cdots, q-1, Z_q=x, \sum_{s=0}^{q-1} \mathbf{1}_{\left(C\left(Z_s\right)=C_i\right)}=n-m \right)\\
    &=\frac{p\left(x, y\right)}{p\left(y, x\right)} S_i^{(n-m)}\left(y, x\right).
\end{split}
\end{align}
The second equality holds because when $x = y^{-}$, the process $Z$ starting from $y$ must remain in $\mT_y$ before entering $x$ for the first time.

Now we come back to calculate $P_\root\left(W_k=x, W_{k+1}=y,\chi_i(e_{k+1})=n,\chi_i(e_k)=m\right)$:
\begin{align*}
    &P_\root\left(W_k = x,\, W_{k+1} = y,\, \chi_i(e_{k+1}) = n,\, \chi_i(e_{k}) = m\right) \\
    &= \sum_{l = m}^{\infty} \sum_{q = n - m}^{\infty} 
       P_\root \left( W_k = x,\, W_{k+1} = y,\, e_k = l,\, e_{k+1} - e_k = q, \chi_i(e_{k}) = m,\, \chi_i(e_{k+1}) - \chi_i(e_{k}) = n - m \right) \\
    &= \sum_{l = m}^{\infty} \sum_{q = n - m}^{\infty} 
       P_y\left( Z_v \in \mT_y \setminus \{y\} \ \forall\, v \geq 1 \right) \\
    &\hspace{.3in}  \times P_x\left( Z_r \in \mT_y \ \text{for } r = 1,\dots,q-1,\ Z_q = y,\
        \sum_{s=1}^q \mathbf{1}_{(C(Z_s) = C_{i})} = n - m\right) \\
    &\hspace{.3in}  \times P_\root\left( Z_l = x,\, \sum_{s = 1}^l \mathbf{1}_{(C(Z_s) = C_i)} = m \right)
       \tag{by \eqref{TTT}} \\
    &= p(y, x)\left( \frac{1}{F(y, x)} - 1 \right) \times \sum_{l = m}^{\infty} P_\root\left( Z_l = x,\, \sum_{s = 1}^l \mathbf{1}_{(C(Z_s) = C_i)} = m \right) \\
    &\hspace{.3in}  \times \sum_{q = n - m}^{\infty} 
       P_x\left( Z_r \in \mT_y \ \text{for } r = 1, \dots, q - 1,\, Z_q = y,\,
       \sum_{s = 1}^q \mathbf{1}_{(C(Z_s) = C_i)} = n - m \right)
       \tag{by \eqref{MMM}} \\[0.6em]
    &= p(y, x)\left( \frac{1}{F(y, x)} - 1 \right) 
       T_i^{(m)}(\root, x) \,
       \frac{p(x, y)}{p(y, x)} \,
       S_i^{(n - m)}(y, x) 
       \tag{by \eqref{SSS} and \eqref{T(x,y)}} \\[0.6em]
    &= \left( \frac{1}{F(y, x)} - 1 \right) 
       T_i^{(m)}(\root, x) \,
       p(x, y) \,
       S_i^{(n - m)}(y, x).
\end{align*}

Collecting the above we have:
\begin{align}\label{TP}
\begin{split}
  P_\root(W_{k+1}=y,&\chi_i(e_{k+1})=n \mid W_k=x,\chi_i(e_k)=m )\\
  &=\frac{P_\root\left(W_{k+1}=y,\chi_i(e_{k+1})=n , W_k=x,\chi_i(e_k)=m\right)}{P_\root\left( W_k=x,\chi_i(e_k)=m\right)}\\
  &= \frac{T_i^{(m)}(\root, x) p(x, y) S_i^{(n-m)}(y, x)\left(\frac{1}{F(y, x)}-1\right)}{T_i^{(m)}(\root, x) p\left(x, x^{-}\right)\left(\frac{1}{F(x, x^-)}-1\right)}\\
  &= \frac{p(x, y)}{p\left(x, x^{-}\right)}\left(\frac{F\left(x, x^{-}\right)}{1-F\left(x, x^{-}\right)}\right)\left(\frac{1-F(y, x)}{F(y, x)}\right) S_i^{(n-m)}(y, x),
\end{split}
\end{align}
which completes the proof.
\end{proof}

\end{document}